\documentclass[12pt]{amsart}
\usepackage{amssymb, mathdots, url,graphicx}
\usepackage[usenames]{color}
\usepackage{enumerate}
\usepackage{mathtools}
\usepackage{xcolor}
\usepackage[hyperfootnotes=false, colorlinks, linkcolor={blue}, citecolor={magenta}, filecolor={blue}, urlcolor={blue}, plainpages=false, pdfpagelabels]{hyperref}

\usepackage{fullpage}

\newcommand{\GL}{\operatorname{GL}}

\newcommand{\SL}{\operatorname{SL}}
\newcommand{\Sp}{\operatorname{Sp}}
\newcommand{\SO}{\operatorname{SO}}
\newcommand{\GSpin}{\operatorname{GSpin}}
\newcommand{\Ad}{\operatorname{Ad}}

\newcommand{\tr}{\operatorname{tr}}

\renewcommand{\Re}{\operatorname{Re}}

\newcommand{\Irr}{\operatorname{Irr}}

\newcommand{\A}{\mathbb{A}}
\newcommand{\Z}{\mathbb{Z}}
\newcommand{\C}{\mathbb{C}}

\newcommand{\Mat}{\operatorname{Mat}}
\newcommand{\Hom}{\operatorname{Hom}}
\newcommand{\Res}{\operatorname{Res}}

\newcommand{\GSp}{\operatorname{GSp}}

\newcommand{\Sym}{\operatorname{Sym}}

\newcommand{\reg}{{\operatorname{reg}}}
\newcommand{\vol}{\operatorname{vol}}

\newcommand{\WD}{\operatorname{WD}}

\newcommand{\JL}{\mathrm{JL}}
\newcommand{\BC}{\mathrm{BC}}
\newcommand{\FJ}{\mathrm{FJ}}
\newcommand{\Asai}{\mathrm{As}}
\newcommand{\elliptic}{\mathrm{ell}}

\newcommand{\End}{\mathrm{End}}
\newcommand{\ul}{\underline}

\newtheorem{theorem}{Theorem}[section]
\newtheorem{lemma}[theorem]{Lemma}

\newtheorem{proposition}[theorem]{Proposition}
\newtheorem{corollary}[theorem]{Corollary}
\newtheorem{conjecture}[theorem]{Conjecture}
\theoremstyle{remark}

\makeatletter
\@namedef{subjclassname@2020}{%
  \textup{2020} Mathematics Subject Classification}
\makeatother

\begin{document}

\title{Epsilon dichotomy for twisted linear models}

\author{Hang Xue}
\address{Department of Mathematics, The University of Arizona, Tucson, AZ 85721, USA}
\email{xuehang@arizona.edu}

\author{Pan Yan}
\address{Department of Mathematics, The University of Arizona, Tucson, AZ 85721, USA}
\email{panyan@arizona.edu}


\date{\today}


\begin{abstract} 
Let $E/F$ be a quadratic extension of local nonarchimedean fields of characteristic zero and let $D$ be a quaternion algebra over $F$ containing $E$. In this paper, we study a relation between the existence of twisted linear models on $\GL_n(D)$ and the local root numbers.
\end{abstract}

\maketitle

\goodbreak 

\tableofcontents

\goodbreak

\section{Introduction}

Let $E/F$ be a quadratic extension of local nonarchimedean fields of characteristic zero and let $\eta=\eta_{E/F}: F^\times / NE^\times \to \{\pm 1\}$ be the quadratic character associated to the extension $E/F$ by local class field theory. 
Let $\omega:F^\times\to \C^\times$ and $\chi:E^\times\to \C^\times$ be two characters satisfying the condition $\chi^n|_{F^\times} \omega=1$.
Let $A$ be a central simple algebra (CSA) over $F$ of dimension $4n^2$ with a fixed embedding $E\to A$, and let $B$ be the centralizer of $E$ in $A$. Then $B$ is a CSA of dimension $n^2$ over $E$. Put $G=A^\times$ and $H=B^\times$, both regarded as algebraic groups over $F$. Let $Z$ be the center of $G$. Let $\pi$ be an irreducible admissible representation of $G$ whose central character is $\omega$. We say that $\pi$ is $(H,\chi^{-1})$-distinguished if
\begin{equation*}
    \Hom_H (\pi, \chi^{-1})\not=0.
\end{equation*}
Here, $\chi^{-1}$ is regarded as a character of $H$ by composing $\chi^{-1}$ with the reduced norm map $H\to E^\times$. 
Elements of $\Hom_H (\pi, \chi^{-1})$ are called (local) twisted linear periods, or (local) twisted linear models.
Let $\pi_0$ be the representation of $\GL_{2n}(F)$ which shares the same $L$-parameter as that of $\pi$ and let $\pi_{0, E}$ be the base change of $\pi_0$ to $\GL_{2n}(E)$.  The following conjecture of Prasad and Takloo-Bighash predicts when $\pi$ is $(H,\chi^{-1})$-distinguished in terms of the local epsilon factor. 

\begin{conjecture}\cite{PrasadTakloo-Bighash2011}
Let the notation be as above. If $\pi$ is $(H,\chi^{-1})$-distinguished, then the following two conditions hold:
\begin{itemize}
    \item[(1)] The Langlands parameter of $\pi_0$ takes values in $\GSp_{2n}(\C)$ with similitude factor $\chi^{-1}|_{F^\times}$.
    
    \item[(2)] $\varepsilon(\pi_{0,E}\otimes\chi)=(-1)^r \eta(-1)^n \chi(-1)^n$, where $r$ is the split rank of $G$.
\end{itemize}
Conversely, if $\pi$ is a discrete series representation and satisfies the above conditions (1) and (2), then $\pi$ is $(H, \chi^{-1})$-distinguished.
\label{conjecture}
\end{conjecture}

Note that in \cite{PrasadTakloo-Bighash2011} it is further assumed that the Jacquet-Langlands transfer $\pi_0=\JL(\pi)$ is generic. This assumption is shown to be unnecessary in \cite{Suzuki2021}. In recent years, there has been much progress towards Conjecture~\ref{conjecture}.  When  $F=\mathbb{R}$, Conjecture~\ref{conjecture} is established in \cite{SuzukiTamori2023} in all cases.
When $F$ is a finite extension of $\mathbb{Q}_p$ and $\chi$ is the trivial character, if either $p\not=2$, or $G = \GL_n(D)$ where $D$ is a quaternion algebra, then Conjecture~\ref{conjecture} holds by a combination of \cite{Secherre}, \cite{Suzuki2021}, \cite{Xue2021} and \cite{SuzukiXue2022}. For a general $\chi$, Conjecture~\ref{conjecture} is proved when $\pi$ is a Steinberg representation, cf.~\cite{Chommaux2019}, and when $G = \GL_{2n}(F)$ and $\pi$ is a depth-zero cuspidal representation, cf.~\cite{ChommauxMatringe2022}.

In this paper, we study Conjecture~\ref{conjecture} for $G = \GL_n(D)$ where $D$ is a quaternion algebra and a general $\chi$.
Our first main result is a proof of the forward direction of the conjecture.

\begin{theorem}[Theorem~\ref{theorem-forward-direction-main}]
Let $G=\GL_n(D)$ where $D$ is a quaternion algebra over $F$ containing $E$, $H=\GL_n(E)$. 
The forward direction of Conjecture~\ref{conjecture} holds.
\label{thm-main-forward}
\end{theorem}

Theorem~\ref{thm-main-forward} is proved in Section~\ref{section-forward}, and is based on a relative trace formula proposed in \cite{XueZhang2022}. We first prove Theorem~\ref{thm-main-forward} when $\pi$ is supercuspidal, by combining a local-global argument and an involution method similar to that of \cite[Section 4]{Xue2021}. If $\pi$ is a discrete series representation, then $\pi$ is a segment of the form
\begin{equation*}
\{\rho\nu_\rho^{-(\ell-1)/2}, \cdots, \rho\nu_\rho^{(\ell-1)/2}\}
\end{equation*}
where $\rho$ is an irreducible supercuspidal representation of $\GL_s(D)$, and $n=s\ell$. Then the discrete series case follows from this consideration and the supercuspidal case. Finally, using a classification of $(H, \chi^{-1})$-distinguished representations, the general case of Theorem~\ref{thm-main-forward} follows from the case of discrete series. 

Our second main result concerns the converse direction of Conjecture~\ref{conjecture}, and we prove it under some additional hypothesis.
\begin{theorem}[Theorem~\ref{thm-converse-for-supercuspidal}]
Let $D$ be a quaternion algebra over $F$ containing $E$.
Let $\pi$ be a discrete series representation of $G=\GL_n(D)$ whose Jacquet-Langlands transfer $\pi_0=\JL(\pi)$ to $\GL_{2n}(F)$ satisfies the conditions (1) and (2) in Conjecture~\ref{conjecture}. Assume that $\chi|_{F^\times}$ is trivial, and that $\BC(\pi_0)$ is supercuspidal. Then $\pi$ is $(H,\chi^{-1})$-distinguished.
\label{thm-main-converse}
\end{theorem}

When $\chi$ is trivial, the counterpart of Theorem~\ref{thm-main-converse} is proved \cite{Xue2021} by studying certain minimal unipotent orbital integrals. The approach we take here, given in Section~\ref{section-converse}, is simpler and avoids the construction of orbital integrals attached to unipotent orbits. The main idea is that under the assumptions in Theorem~\ref{thm-main-converse}, we can globalize $\pi_0$ to a cuspidal automorphic representation $\ul \pi_0$ of $\GL_{2n}(\A_{\ul F})$ whose base change $\BC(\ul \pi_0)$ to $\GL_{2n}(\A_{\ul E})$ is globally distinguished by $(\GL_{n}(\A_{\ul E})\times \GL_{n}(\A_{\ul E}), \ul \chi_{H^\prime}^{-1})$ and $(\GL_{2n}(\A_{\ul F}), \ul\eta)$. 
Here, we denote a global object by a letter with an underline, and $\ul \chi_{H^\prime}^{-1}$ is the character on $\GL_{n}(\A_{\ul E})\times \GL_{n}(\A_{\ul E})$ defined by $\ul \chi_{H^\prime}^{-1}(h_1, h_2)=\ul \chi^{-1}(h_1 \overline{h_2})$.
We refer the reader to Section~\ref{section-converse} for unexplained notations. The main difficulty is that we require $\BC(\ul \pi_0)$ to be distinguished by both $(\GL_{n}(\A_{\ul E})\times \GL_{n}(\A_{\ul E}), \ul \chi_{H^\prime}^{-1})$ and $(\GL_{2n}(\A_{\ul F}), \ul\eta)$. To solve the problem, we move on to the Bessel periods for orthogonal groups and use a trace formula argument.

The condition that $\chi|_{F^\times}$ is trivial will be used only at this final globalization step. With this condition, we may view $\chi$ as a character of $\SO(2)$ and make use some known cases of the global Gross--Prasad conjecture for $\SO(2n+1) \times \SO(2)$. To treat the case of a general $\chi$, we will need a Gross--Prasad type conjecture for $\GSpin(2n+1) \times \GSpin(2)$. We hope that the argument in this paper could stimulate research in this direction.

We stick to the case $G = \GL_n(D)$ in this paper because we follow a relative trace formula approach and the relevant relative trace formula is only established when $G = \GL_n(D)$. Once we have the relative trace formula for all central simple algebras, it is possible to extend our argument to the more general situation.

We end the introduction by giving a brief overview of the structure of the rest of the paper. In Section~\ref{section-preliminaries}, we review some basic facts about representations over local nonarchimedean field, and several local and global functorial lifts. 
In Section~\ref{section-geometric-side} and Section~\ref{section-spectral-side}, we recall the geometric side and the spectral side of the relative trace formula of \cite{XueZhang2022} respectively. Then the proof of Theorem~\ref{thm-main-forward} is given in Section~\ref{section-forward}, while the proof of Theorem~\ref{thm-main-converse} is given in Section~\ref{section-converse}. 

\subsection*{Acknowledgements}
We thank the referee for the careful reading and helpful suggestions, and especially for raising the question of the multiplicity one theorem for twisted split linear periods for $(\GL_{2n}(E)/\GL_n(E)\times\GL_n(E),\chi_{H^\prime}^{-1})$, which have helped us to improve the paper. 
HX is partially supported by the NSF grant DMS \#2154352.
PY is partially supported by an AMS-Simons Travel Grant.

\section{Preliminaries}
\label{section-preliminaries}

In this section, we let $F$ be either a number field or a local nonarchimedean field of characteristic zero. Let $E/F$ be a quadratic field extension, and let $C$ be a central division algebra of dimension $d^2$ over $F$. Let $G_r=\GL_r(C)$ be the multiplicative group of $\Mat_r(C)$. If $F$ is a number field, we denote by $\A_F$ its ring of adeles. 

\subsection{Representations over local nonarchimedean field}
Let $F$ be a local nonarchimedean field of characteristic zero. We first recall some basic facts about the local Langlands correspondence and the local Jacquet-Langlands transfer for $G_r=\GL_r(C)$. We refer the reader to \cite{AubertBaumPlymenSolleveld2016, DeligneKazhdanVigneras1984}  for more details. 

An element $g\in \GL_{rd}(F)$ is called regular semisimple if the characteristic polynomial of $g$ has distinct roots in an algebraic closure of $F$. There is a standard way of defining the characteristic polynomial for elements of $\GL_r(C)$ (see for example \cite{Pierce1982}). If $g^\prime\in \GL_{r}(C)$, then the characteristic polynomial of $g^\prime$ has coefficients in $F$, and it is monic and has degree $rd$. The definition of a regular semisimple element of $\GL_r(C)$ is the same as for $\GL_{rd}(F)$. For $g\in \GL_r(C)$ and $g^\prime\in  \GL_{rd}(F)$, we say that $g$ corresponds to $g^\prime$ if $g$ and $g^\prime$ are regular semisimple and have the same characteristic polynomial. In this case, we write  $g\leftrightarrow g^\prime$. Let $\Irr(\GL_r(C))$ (resp. $\Irr(\GL_{rd}(F))$ denote the set of equivalence classes of irreducible admissible representations of $\GL_r(C)$ (resp. $\GL_{rd}(F)$), and let
$\Irr_{\mathrm{disc}}(\GL_r(C))$ (resp. $\Irr_{\mathrm{disc}}(\GL_{rd}(F))$ denote the subset of discrete series representations of $\GL_r(C)$ (resp. $\GL_{rd}(F)$). For a representation $\pi$ of $\GL_r(C)$ or $\GL_{rd}(F)$, we write $\theta_\pi$ for its character. 

\begin{theorem}\cite{DeligneKazhdanVigneras1984}
There is a unique bijection $\JL: \Irr_{\mathrm{disc}}(\GL_r(C)) \to \Irr_{\mathrm{disc}}(\GL_{rd}(F))$ such that for $\pi\in \Irr_{\mathrm{disc}}(\GL_r(C))$, we have
\begin{equation*}
(-1)^r \theta_{\pi}(g)=(-1)^{rd} \theta_{\JL(\pi)}(g^\prime)
\end{equation*}
for all $g\in \GL_r(C)$ and $g^\prime\in  \GL_{rd}(F)$ such that $g\leftrightarrow g^\prime$. 
\end{theorem}

Let $\WD_F=W_F\times\SL_2(\C)$ be the Weil-Deligne group of $F$. Let $\Phi(G_r)$ be the set of equivalence classes of $L$-parameters $\phi:\WD_F\to \GL_{rd}(\C)$ which is relevant to $G_r$. Note that $\Phi(G_r)\subsetneq \Phi(\GL_{rd}(F))$ if $G_r$ is not split. 

The local Langlands correspondence (LLC) for $\GL_{rd}(F)$ established in \cite{HarrisTaylor2001, Henniart2000, Scholze2013} together with the Jacquet-Langlands transfer gives the LLC for $G_r$, which is a canonical bijective map
\begin{equation*}
\mathrm{Irr}(G_r)\to \Phi(G_r)	.
\end{equation*}
It follows that we have a canonical injective map
\begin{equation*}
\text{rec}_{C,r}: \mathrm{Irr}(G_r) \to \Phi(\GL_{rd}(F)),
\end{equation*}
whose image is $\Phi(G_r)$. Note that for each $\phi\in \Phi(\GL_{rd}(F))$, its fiber $\text{rec}_{C,r}^{-1}(\phi)$ is a singleton if $\phi\in \Phi(G_r)$, and empty otherwise. Given an irreducible representation $\pi$ of $G_r$, we call the the corresponding representation of $\GL_{rd}(F)$ which shares the same $L$-parameter as that of $\pi$ the Jacquet-Langlands transfer of $\pi$, and denote it by $\JL(\pi)$.

 Let $E/F$ be a quadratic field extension. Given an irreducible admissible representation $\pi$ of $\GL_n(F)$, we denote by $\BC(\pi)$ the local quadratic base change of $\pi$ defined in \cite[\S 1.6]{ArthurClozel1989}, which is an irreducible admissible representation of $\GL_n(E)$.

Let $\tau$ be an irreducible admissible representation of the split $\SO_{2n+1}(F)$ ($\mathrm{GSpin}_{2n+1}(F)$ respectively) with parameter $\phi_\tau:\WD_F\to \Sp_{2n}(\C)$ ($\WD_F\to \GSp_{2n}(\C)$ respectively), and let $r:\Sp_{2n}(\C)\to \GL_{2n}(\C)$ ($r:\GSp_{2n}(\C)\to \GL_{2n}(\C)$ respectively)  be the embedding. An irreducible admissible representation $\pi$ of $\GL_{2n}(F)$ is called a local functorial lift of $\tau$ if the parameter $\phi_\pi$ of $\pi$ is given by $\phi_{\pi}=r\circ \phi_\tau$.

\subsection{Automorphic representations}
Let $E/F$ be a quadratic extension of number fields with $\eta=\eta_{E/F}$ the quadratic character associated to the extension $E/F$ via global class field theory. The global Jacquet-Langlands transfer is an injective map from the set of irreducible discrete series representations of $G_r(\A_F)$ to that of $\GL_{rd}(\A_F)$ \cite[Theorem 5.1]{Badulescu2008}. We denote this map still by $\JL$, since there is no chance of confusion. 
For the following fact about global base change lift, we refer the reader to \cite[\S 3, Theorem 4.2]{ArthurClozel1989}. Let $\pi$ be an irreducible cuspidal automorphic representation of $\GL_n(\A_F)$ such that $\pi\not\cong \pi\times \eta$. Then its base change to $\GL_{n}(\A_E)$, denoted by $\BC(\pi)$, exists and is unique, which is an irreducible cuspidal automorphic representation of $\GL_n(\A_E)$. Moreover, by \cite[\S 3, Theorem 5.1]{ArthurClozel1989},  $\BC(\pi)_v = \BC(\pi_v)$ for all places $v$ of $F$. 

Let $\tau$ be an irreducible cuspidal automorphic representation of $\SO_{2n+1}(\A_F)$ and let $\pi$ be an irreducible automorphic representation of $\GL_{2n}(\A_F)$. We say that $\pi$ is a (weak) functorial lift of $\tau$ if for almost all finite places $v$ of $F$ where $\tau_v$ is unramified, $\pi_v$ is the local functorial lift of $\tau_v$, and at every infinite place $v$, the infinitesimal character of $\pi_v$ is determined, via the $L$-morphism, by that of $\tau_v$.  The existence of the global functorial lift from $\SO_{2n+1}$ to $\GL_{2n}$ is proven to hold by Arthur \cite{Arthur2013} using the trace formula, and independently by Cai, Friedberg and Kaplan \cite{CaiFriedbergKaplan2024} using the generalized doubling method and the converse theorem.

\section{The geometric side: orbital integrals and smooth transfer}
\label{section-geometric-side}

In this section, we assume that $E/F$ is a quadratic field extension of either local fields or global fields of characteristic zero with the nontrivial Galois involution $g\mapsto \bar{g}$.

\subsection{The split side}
\label{subsection-geom-side-split}
Let $E/F$ be either local or global.
Let $G^\prime=\Res_{E/F}(\GL_{2n})$, $H^\prime=\Res_{E/F}(\GL_n\times\GL_n)$ which is embedded in $G^\prime$ as diagonal blocks, and $H^{\prime\prime}=\GL_{2n}(F)$. The split center of $G^\prime$ is  $Z^\prime\cong \GL_{1, F}$, embedded in $G^\prime$ diagonally. Let 
\begin{equation*}
    S^\prime=\left\{ g \overline{g}^{-1} | g\in G^\prime \right\} \cong G^\prime/ H^{\prime\prime}.
\end{equation*}
This is a symmetric space on which $H^\prime$ acts by twisted conjugation. We say that an element in $S^\prime$ is semisimple if its $H^\prime$-orbit is Zaraski closed. We say that an element in $S^\prime$ is regular semisimple if its stabilizer in $H^\prime$ is a torus of dimension $n$, and it is in addition elliptic if its stabilizer is an anisotropic torus modulo the split center of $G^\prime$. An element $g\in G^\prime$ is semisimple (regular semisimple, elliptic, respectively) if $g\overline{g}^{-1}$ is so in $S^\prime$. Let $G^\prime_\reg$ and $G^\prime_\elliptic$ denote the subsets of regular semisimple and elliptic elements of $G^\prime$. 
We have the following classification of semisimple elements in $S^\prime(F)$.

\begin{lemma}\cite{XueZhang2022}
The following results hold.
\begin{itemize}
    \item[(1)] Every semisimple element in $S^\prime$ is in the $H^\prime$-orbit of the form $s^\prime(\alpha, n_1, n_2, n_3)$, where $n_1+n_2+n_3=n$ is a partition of $n$, $\alpha\in \GL_{n_1}(E)$, $\alpha\overline{\alpha}\in \GL_{n_1}(F)$ is semisimple in the usual sense, $\det(\alpha\overline{\alpha}-1)\not=0$, and 
    \begin{equation*}
    s^\prime(\alpha, n_1, n_2, n_3)=
    \left( 
    \begin{smallmatrix}
    \alpha &  & &1_{n_1} & & \\
    &0_{n_2} & & &1_{n_2} & \\
    & &1_{n_3} & & &0_{n_3} \\
    1_{n_1}-\alpha\overline{\alpha}& & &-\overline{\alpha} & & \\
    &1_{n_2} & & &0_{n_2} & \\
    & &0_{n_3} & & &1_{n_3} \\
    \end{smallmatrix}
    \right).
    \end{equation*}
    
    \item[(2)] Let $s^\prime=s^\prime(\alpha, n_1, n_2, n_3)$ be a semisimple element as in (1). It is regular semisimple if and only if $n_1=n$ and $\alpha\overline{\alpha}\in \GL_{n}(F)$ is regular semisimple in the usual sense. It is elliptic if $\alpha\overline{\alpha}\in \GL_n(F)$ is elliptic in the usual sense. 
    
    \item[(3)] Two regular semisimple $s^\prime(\alpha_1, n, 0, 0)$ and $s^\prime(\alpha_2, n, 0, 0)$ are in the same $H^\prime$-orbit if and only if $\alpha_1$ and $\alpha_2$ are twisted conjugate in $\GL_n(E)$.
\end{itemize}
\end{lemma}

If $n_2=n_3=0$, we write $s^\prime(\alpha, n, 0, 0)=s^\prime(\alpha)$.

Let $\mathbf{A}^n$ be the affine space of dimension $n$ over $F$ and let $q^\prime:S^\prime\to \mathbf{A}^n$ be the morphism
\begin{equation*}
\begin{pmatrix}
A &B\\ C&D
\end{pmatrix} \to \tr \wedge^i (2A\overline{A}-1),\quad i=1, \cdots, n.
\end{equation*}
Then the map $q^\prime$ is a categorical quotient \cite[Lemma 5.4]{XueZhang2022}.

Now we assume $E/F$ is a quadratic extension of number fields, and let $\eta:F^\times\backslash \A_F^\times\to \C^\times$ be the quadratic character attached to $E/F$ by class field theory. We fix a character $\widetilde{\eta}:E^\times\backslash \A_E^\times\to \mathbb{C}^\times$ such that $\widetilde{\eta}|_{\A_F^\times}=\eta$.  Let $f^\prime\in \mathcal{C}_c^\infty(G^\prime(\A_F))$, and let $x\in G^\prime_\reg(\A_F)$ be regular semisimple. We define the orbital integral
\begin{equation*}
    O^{G^{\prime}}(x, f^{\prime}) :=\int_{\left(H^{\prime} \times H^{\prime \prime}\right)_x\left(\A_F\right) \backslash\left(H^{\prime} \times H^{\prime \prime}\right)\left(\A_F\right)} f^{\prime} (h^{-1} x h^{\prime \prime} )\left(\chi_{H^{\prime}}  \chi^{-1} \widetilde{\eta}^{-1}\right)(h)(\chi \widetilde{\eta})^{-1} (h^{-1} x h^{\prime \prime} ) \mathrm{d} h \mathrm{~d} h^{\prime \prime}.
\end{equation*}
Here, for $h=\begin{pmatrix} h_1 & \\ & h_2\end{pmatrix}\in H^\prime(F)$, $\chi_{H^\prime}(h)=\chi(h_1\overline{h_2})$. This orbital integral is absolutely convergent for all regular semisimple $x$ (see \cite[Appendix A]{XueZhang2022}).

The orbital integral $O^{G^{\prime}}(g, f^{\prime})$ can be simplified in the following way. Put
\begin{equation*}
    \tilde{f}^{\prime}(g \bar{g}^{-1}):=\int_{H^{\prime \prime}\left(\A_F\right)} f^{\prime}(g h)(\chi \widetilde{\eta})^{-1} (g h) \mathrm{d} h .
\end{equation*}
Then $\tilde{f}^{\prime}\in \mathcal{C}_c^\infty(S^\prime(\A_F))$, and 
\begin{equation*}
    O^{G^{\prime}} (g, f^{\prime} )=O^{S^{\prime}}(s^{\prime}, \tilde{f}^{\prime} ):=\int_{H_{s^{\prime}}^{\prime}\left(\mathbb{A}_F\right) \backslash H^{\prime}\left(\mathbb{A}_F\right)} \tilde{f}^{\prime} (h^{-1} s^{\prime} \bar{h} )\left(\chi_{H^{\prime}} \chi^{-1}\widetilde{\eta}^{-1}\right)(h) \mathrm{d} h, \quad s^{\prime}=g \bar{g}^{-1} .
\end{equation*}

If $v$ is a place of $F$, we define the local orbital integral by the same formula, but integrating over $F_v$-points instead, and the local orbital integral can be simplified in a similar way.

We now define a transfer factor at each place $v$ for regular semisimple elements of $G^\prime(F)$. 
We fix a purely imaginary element $\tau\in E^\times$ such that $\overline{\tau}=-\tau$. 
Let $x\in G^\prime(F_v)$ be regular semisimple, and write 
\begin{equation*}
    x \bar{x}^{-1}=\left(\begin{array}{ll}
\alpha_1 & \alpha_2 \\
\alpha_3 & \alpha_4
\end{array}\right) \in S^{\prime}(F_v)
\end{equation*}
with $\alpha_i\in \Mat_n(E_v)$. Then we define
\begin{equation*}
    \kappa^{G^{\prime}}_v(x)=\chi_v (\alpha_4 ) \widetilde{\eta}_v (\tau \alpha_2).
\end{equation*}

\subsection{The non-split side}
\label{subsection-geom-side-nonsplit}
Assume $E/F$ is global.
Let $D$ be a quaternion algebra over $F$ with fixed embedding $E\to D$. Let $G=\GL_n(D)$, $H=\Res_{E/F}\GL_{n}$, and $Z=\GL_{1, F}$. We fix an element $\epsilon\in NE^\times$ ($F^\times\backslash NE^\times$ respectively) if $D$ splits (ramified respectively) and the group $G$ is realized as a subgroup of $\GL_{2n}(E)$ consisting of elements of the form
\begin{equation*}
    \begin{pmatrix}
    A & \epsilon B\\
    \overline{B} &  \overline{A}
    \end{pmatrix}, \quad A, B\in \Mat_n(E).
\end{equation*}
Then $H$ consists of matrices of the form $\displaystyle \begin{pmatrix} A & \\ & \overline{A}\end{pmatrix}$, $A\in \GL_n(E)$.

For $g\in \GL_{2n}(E)$, we define an involution
\begin{equation*}
    \theta(g)=\left(\begin{array}{ll}
1_n & \\
& -1_n
\end{array}\right) g\left(\begin{array}{ll}
1_n & \\
& -1_n
\end{array}\right).
\end{equation*}
Then $H$ is the group of fixed points of $\theta$. We consider the following symmetric space
\begin{equation*}
    S=\{g\theta(g)^{-1}|g\in G\}
\end{equation*}
and then $H$ acts on $S$ via conjugation. We have an action of $H\times H$ on $G$ by $(h_1, h_2)\cdot y= h_1^{-1}yh_2$. 
Similar to $S^\prime$, we say that an element in $S$ is semisimple if its $H$-orbit is Zariski closed. It is regular semisimple if its stabilizer in $H$ is a torus of dimension $n$. It is elliptic if furthermore its stablizer in $H$ is an elliptic torus modulo $Z$ (over $F$). An element $g\in G$ is semisimple (regular semisimple, elliptic, respectively) if $g\theta(g)^{-1}$ is so in $S$. Let $G_\reg$ and $G_\elliptic$ denote the subsets of regular semisimple and elliptic elements of $G$. 

\begin{lemma}\cite{Guo1996}
Every semisimple $g\in G(F)$ is in the $(H\times H)(F)$-orbit of 
\begin{equation*}
g\left(\beta, n_1, n_2, n_3\right)=\left(\begin{array}{cccccc}
1_{n_1} & & & \epsilon \beta & & \\
& 0_{n_2} & & & \epsilon 1_{n_2} & \\
& & 1_{n_3} & & & 0_{n_3} \\
\bar{\beta} & & & 1_{n_1} & & \\
& 1_{n_2} & & & 0_{n_2} & \\
& & 0_{n_3} & & & 1_{n_3}
\end{array}\right),    
\end{equation*}
where $\beta\in \GL_{n_1}(E)$, $\beta\overline{\beta}\in \GL_{n_1}(F)$, and $\det(1-\epsilon \beta\overline{\beta})\not=0$. It is regular semisimple (elliptic respectively) if $n_2=n_3=0$ and $\beta\overline{\beta}$ is regular semisimple (elliptic respectively) in $\GL_n(F)$ in the usual sense. 
\end{lemma}

If $n_2=n_3=0$, we write $s(\beta,  n, 0, 0)=s(\beta)$ and $g(\beta, n, 0, 0)=g(\beta)$.

Let $g\in G$ and let $s=g\theta(g)^{-1}=\begin{pmatrix} A & B\\ C & D\end{pmatrix}\in S$. We define a morphism $q:G\to \mathbf{A}^n$ by
\begin{equation*}
    g\mapsto \tr \wedge^i A,\quad i=1, \cdots, n.
\end{equation*}
Then $q$ is a categorical quotient \cite[Lemma 5.7]{XueZhang2022}. 

Now we define orbital integrals. Let $y\in G_\reg(F)$ be regular semisimple and let $f\in \mathcal{C}_c^\infty(G(\A_F))$. We define the orbital integral
\begin{equation*}
    O^G(y, f):=\int_{(H \times H)_y(\A_F) \backslash(H \times H)(\A_F)} f(h_1^{-1} y h_2) \chi(h_1^{-1} h_2)^{-1} \mathrm{d} h_1 \mathrm{d} h_2.
\end{equation*}
This orbital integral is absolutely convergent for all regular semisimple $y$.

This orbital integral can be simplified as follows. Put
\begin{equation*}
    \tilde{f}\left(g \theta(g)^{-1}\right)=\int_{H(\A_F)} f(g h) \chi(g h) ^{-1}\mathrm{d} h.
\end{equation*}
Then one can check that $\tilde{f}\in \mathcal{C}_c^\infty(S(\A_F))$ and
\begin{equation*}
     O^G(g, f)= O^S(s, \tilde{f}):=\int_{H_s(\A_F) \backslash H(\A_F)} \tilde{f}\left(h^{-1} s h\right) \mathrm{d} h, \quad s=g\theta(g)^{-1}.
\end{equation*}

If $v$ is a place of $F$, the local orbital integral is defined similarly, and it can be simplified in a similar way. 

The transfer factor is defined as follows. Let $y\in G_\reg(F_v)$ be a regular semisimple element, and write $y^{-1}=\begin{pmatrix}
    y_1 & \epsilon y_2 \\ 
    \overline{y_2} & \overline{y_1}
    \end{pmatrix}$. Then at each place $v$ we define
\begin{equation*}
    \kappa^G_v(y)=\chi_v(y_1).
\end{equation*}

\subsection{Matching of test functions}

Let $F$ be global. 
Let $x\in G^\prime(F)$ and let $y\in G(F)$ be regular semisimple elements, and write
\begin{equation*}
     x \bar{x}^{-1}=\left(\begin{array}{ll}
\alpha_1 & \alpha_2 \\
\alpha_3 & \alpha_4
\end{array}\right) \in S^{\prime}(F), \quad y\theta(y)^{-1}=\left(\begin{array}{ll}
\beta_1 & \beta_2 \\
\beta_3 & \beta_4
\end{array}\right) \in S(F),
\end{equation*}
where $\alpha_i, \beta_i\in \Mat_n(E)$. 
We say that $x$ and $y$ match if $2\alpha_1 \overline{\alpha_1}-1$ and $\beta_1$ have the same characteristic polynomial.
 We note that not all regular semisimple $x\in G^\prime(F)$ match a regular semisimple $y\in G(F)$, and vice versa. The definition is similar when the field is local. We also note that there is a neighborhood of $1\in G(F_v)$ such that every regular semisimple $y$ in this neighborhood matches some $x\in G^\prime(F_v)$.

For each place $v$ of $F$, we define 
\begin{equation*}
    \mathcal{C}_c^\infty(G^\prime(F_v))_0=\{ f^\prime\in \mathcal{C}_c^\infty(G^\prime(F_v))| O^{G^\prime}(x, f_v^\prime)=0 \text{ for all $x$ not matching any $y\in G(F_v)$} \},
\end{equation*}
and $\mathcal{C}_c^\infty(G(F_v))_0$ in a similar way. We say that two test functions $f^\prime\in \mathcal{C}_c^\infty(G^\prime(F_v))_0$ and $f\in \mathcal{C}_c^\infty(G(F_v))_0$ match if 
\begin{equation*}
    \kappa_v^{G^\prime}(x) O^{G^\prime}(x, f_v^\prime) =  \kappa_v^{G}(y) O^{G}(y, f_v)
\end{equation*}
for all matching regular semisimple elements $x\in G^\prime(F_v)$ and $y\in G(F_v)$. Two test functions $f^\prime=\otimes f_v^\prime\in \mathcal{C}_c^\infty(G^\prime(\mathbb{A}_F))$ and $f=\otimes f_v\in \mathcal{C}_c^\infty(G(\mathbb{A}_F))$ match if $f_v^\prime\in \mathcal{C}_c^\infty(G^\prime(F_v))_0$ and $f_v\in \mathcal{C}_c^\infty(G(F_v))_0$ and they match for all places $v$ of $F$.

We have the following results concerning the geometric side of the trace formula.

\begin{theorem}\cite[Theorem 2.1]{XueZhang2022}
Assume that $v$ is nonarchimedean and nonsplit. For any $f_v^\prime\in \mathcal{C}_c^\infty(G^\prime(F_v))_0$, there is an $f_v\in \mathcal{C}_c^\infty(G(F_v))_0$ that matches it, and vice versa.
\label{GeometricSide-matching-thm1}
\end{theorem}

\begin{theorem}\cite[Theorem 2.2]{XueZhang2022}
Let $v$ be a nonsplit nonarchimedean odd place of $F$. Assume the quaternion algebra $D$ splits at $v$ and $\chi_v$ is unarmified at $v$. Let $\mathfrak{o}_{v}$ be the ring of integers of $F_v$. We pick the measures on $G^\prime(F_v)$ and $G(F_v)$ so that the volumes of $G^\prime(\mathfrak{o}_v)$ and $G(\mathfrak{o}_v)$ are equal to 1. Then $\mathbf{1}_{G^\prime(\mathfrak{o}_v)}$ and $\mathbf{1}_{G(\mathfrak{o}_v)}$ match.
\label{GeometricSide-Fundamental-Lemma}
\end{theorem}

At the split places of $F$, the matching of test functions can be made explicit, as we explain below. Assume $v$ is a split place of $F$. Then $G(F_v)=\GL_{2n}(F_v)$ and $G^\prime(F_v)=\GL_{2n}(F_v)\times\GL_{2n}(F_v)$. We fix a measure on $\GL_{2n}(F_v)$ and then we have measures on $G^\prime(F_v)$ and $G(F_v)$ under these identifications. The character $\eta_v$ is trivial, so $\widetilde{\eta}_v$ takes the form $(\eta_0, \eta_0^{-1})$ where $\eta_0$ is a character of $F_v^\times$. The character $\chi_v$ is of the form $(\chi_1, \chi_2)$ where $\chi_1, \chi_2$ are characters of $F_v^\times$. Two regular semisimple elements $y\in G(F_v)$ and $(x_1, x_2)\in G^\prime(F_v)$ match if $y=x_1 x_2^{-1}$. Let $f^\prime=(f^\prime_1, f^\prime_2)\in \mathcal{C}_c^\infty(G^\prime(F_v))\cong \mathcal{C}_c^\infty(G(F_v))\otimes \mathcal{C}_c^\infty(G(F_v))$, and put
\begin{equation*}
    f(g)=\int_{\GL_{2n}(F_v)}f_1^\prime(gh)f_2^\prime(h)\chi_1(h)^{-1}\chi_2(h)^{-1}dh, \quad g\in \GL_{2n}(F_v).
\end{equation*}
Then the functions $f^\prime$ and $f$ match \cite[Lemma 2.3]{XueZhang2022}.

\section{The spectral side: spherical characters}
\label{section-spectral-side}

In this section, we assume that $E/F$ is a quadratic extension of non-archimedean local fields of characteristic zero and $\pi$ is a $(H,\chi^{-1})$-distinguished supercuspidal representation of $G$.

\subsection{The split side}
\label{subsection-spectral-split-side}
Let $\Pi$ be an irreducible generic (unitary) representation of $G^\prime$. 
We denote $\Pi^\vee$ the contragredient of $\Pi$, and $\Pi^c$ the Galois conjugate of $\Pi$ relative to $E/F$ (i.e., $\Pi^c(g)=\Pi(\overline{g})$). It is known that $\Pi^c\cong \Pi$ if and only if $\Pi=\BC(\tau)$ for some irreducible admissible representation $\tau$ of $\GL_{2n}(F)$ \cite{ArthurClozel1989}.

Let us recall the notion of good characters of $H^\prime$ following \cite{ChenSun2020}. We use $|\cdot|_E$ to denote the normalized absolute value on $E$, and also use it to denote the character $t\mapsto |t|_E$ of $E^\times$. We say that a character $\xi:\GL_n(E)\to \mathbb{C}^\times$ is  good if it is equal to $\alpha\circ\det$ for some character $\alpha$ of $E^\times$ such that the character $\alpha^{2r}\cdot |\cdot|_E^{-m}$ is not trivial for all $r\in \{\pm 1, \pm2, \cdots, \pm n\}$ and for all $m\in \{1, 2, \cdots, 2n^2\}$. Note that $\xi$ is good if and only if $\xi^{-1}$ is good, and that all but finitely many characters of $\GL_n(E)$ are good. A character $\xi_0\otimes\xi_1$ of $\GL_n(E)\times\GL_n(E)$ is said to be good if the character $\xi_0\xi_1^{-1}$ of $\GL_n(E)$ is good.

We say that $\Pi$ is $(H^\prime, \chi_{H^\prime}^{-1})$-distinguished if $\Hom_{H^\prime}(\Pi\otimes \chi_{H^\prime}, \mathbb{C})\not=0$. 
Note that the character $\chi_{H^\prime}$ is good because the restriction of $\chi{\overline{\chi}^{-1}}$ to $\GL_n(F)$ is trivial.
Then the space $\Hom_{H^\prime}(\Pi\otimes \chi_{H^\prime}, \mathbb{C})$ is at most one-dimensional \cite{ChenSun2020}.
We say that $\Pi$ is $(H^{\prime\prime}, \chi^{-1}\eta)$-distinguished if $\Hom_{H^{\prime\prime}}(\Pi\otimes \chi\eta, \mathbb{C})\not=0$. It's known that the space $\Hom_{H^{\prime\prime}}(\Pi\otimes \chi\eta, \mathbb{C})$ is at most one-dimensional \cite{Flicker1991}.

Let $Q$ be the Shalika subgroup of $\GL_{2n}(E)$ which consists of matrices of the form
\begin{equation*}
    \begin{pmatrix}
    g & \\ & g
    \end{pmatrix} 
    \begin{pmatrix}
    1 & x \\ & 1
    \end{pmatrix} , \, \, g\in \GL_n(E), x\in \Mat_n(E).
\end{equation*}
Let $\psi:E\to \mathbb{C}^\times$ be a non-trivial additive character and $ \xi:E^\times\to \mathbb{C}^\times$ a multiplicative character. We define a character $\theta:Q(E)\to \mathbb{C}^\times$ by
\begin{equation*}
    \theta\left(  \begin{pmatrix}
    g & \\ & g
    \end{pmatrix} 
    \begin{pmatrix}
    1 & x \\ & 1
    \end{pmatrix}\right)=\psi(\tr(x))\xi(\det(g)).
\end{equation*}
Let $\Pi$ be an irreducible admissible representation of $\GL_{2n}(E)$ with central character $\omega$, and assume that $\xi^n \omega=1$. A functional $\lambda: V_\Pi \to \mathbb{C}$ is called a Shalika functional if $\lambda$ satisfies
\begin{equation*}
    \lambda(\Pi(s)v)=\theta^{-1}(s)\lambda(v)
\end{equation*}
for every $s\in Q$ and $v$ in the space of $\Pi$. It is known that the space of Shalika functionals on $\Pi$ is at most one-dimensional \cite{JacquetRallis1996}.

Assume $\Pi$ has a non-zero Shalika functional, and let $\mathcal{V}$ be the corresponding Shalika model of $\Pi$ corresponding to the characters $\psi$ and $\xi=(\chi\chi^c)$.
Let $\mathcal{W}$ be the Whittaker model of $\Pi$, defined with respect to the character $\psi$. We fix an isomorphism
\begin{equation*}
    \mathcal{W}\to \mathcal{V}, \, \, W\mapsto \phi_W.
\end{equation*}
We also let $\widetilde{\mathcal{W}}$ and $\widetilde{\mathcal{V}}$ be the Whittaker model and Shalika model of $\Pi^\vee$, defined by the character $\psi^{-1}$ and $\psi^{-1}, (\chi\chi^c)^{-1}$ respectively. We also fix an isomorphism
\begin{equation*}
    \widetilde{\mathcal{W}} \to \widetilde{\mathcal{V}}, \, \, W\mapsto \phi_W.
\end{equation*}
We consider the local Friedberg-Jacquet integral
\begin{equation*}
    Z^{\FJ}(s, \phi, \chi)=\int_{\GL_n(E)} \phi\left(\begin{pmatrix} a & \\ & 1\end{pmatrix}\right) \chi(a) |\det(a)|^{s-\frac{1}{2}}da, \, \phi\in \mathcal{V}.
\end{equation*}
By \cite[Proposition 3.1]{FriedbergJacquet1993}, this integral converges absolutely for $\Re(s)$ large enough and has a meromorphic continuation to $\mathbb{C}$. Also, it is a holomorphic multiple of $L(s, \Pi\otimes\chi)$. Moreover, there is a $\phi\in \mathcal{V}$ such that the integral is equal to $L(s, \Pi\otimes\chi)$. Thus, 
\begin{equation*}
    W \mapsto \ell^\prime(W)=Z^{\FJ}(\frac{1}{2}, \phi_W, \chi)
\end{equation*}
defines a non-zero element in $\Hom_{H^\prime}(\Pi\otimes \chi_{H^\prime}, \mathbb{C})$.

Let $P^\prime$ be the mirabolic subgroup of $G^\prime$ and $N^\prime$ the standard upper triangular unipotent subgroup of $G^\prime$. 
Define
\begin{equation*}
    \ell^{\prime\prime}(W)=\int_{N^\prime\cap H^{\prime\prime}\backslash P^\prime\cap H^{\prime\prime}}W(h) (\chi \eta)(h)dh.
\end{equation*}
Notice that the character $\chi\eta$ is trivial on $N^\prime\cap H^{\prime\prime}$.
The above integral is absolutely convergent and hence define a non-zero element in  $\Hom_{H^{\prime\prime}}(\Pi\otimes \chi  \eta, \mathbb{C})$.

We define a local spherical character on the split side as follows. For any $f^\prime\in \mathcal{C}_c^\infty(G^\prime)$, put
\begin{equation*}
    I_{\Pi}(f^\prime)=\sum_{W} \ell^\prime(\Pi(f^\prime) W))\overline{\ell^{\prime\prime}(W)}
\end{equation*}
where $W$ ranges through an orthonormal basis of $\mathcal{W}$. This is called the spherical character of $\Pi$ relative to $(H^\prime,  \chi_{H^\prime}^{-1})$ and $ (H^{\prime\prime}, (\chi\eta)^{-1})$. 
A basic property of the spherical character $I_\Pi$ is the representability by a locally integrable function, stated below. 

\begin{proposition}
\label{prop-representability}
    $I_{\Pi}$ is represented by a locally integrable function $\Theta_\Pi$ on $G^\prime$ which is locally constant on the regular semisimple locus, and satisfies
    \begin{equation*}
        \Theta_\Pi( h g h^{\prime\prime})=\chi_{H^\prime}(h)(\chi \eta)(h^{\prime\prime})\Theta_\Pi(g), \quad h\in H^\prime(F), h^{\prime\prime}\in H^{\prime\prime}(F).
    \end{equation*}
\end{proposition}

We will sketch a proof of Proposition~\ref{prop-representability} in Appendix~\ref{section-ellipticity-of-supercuspidal}.

We say that $\Pi$ is $(H^\prime, H^{\prime\prime})$-elliptic if there is an elliptic regular semisimple element $x\in G^\prime$ such that $x$ matches some $y\in G$ and $\Theta_\Pi(x)\not=0$.

In Appendix~\ref{section-ellipticity-of-supercuspidal}, we will prove the following.
\begin{theorem}
Assume $\Pi$ is supercuspidal, $\Hom_{H^\prime}(\Pi\otimes \chi_{H^\prime}, \mathbb{C})\not=0$ and $\Hom_{H^{\prime\prime}}(\Pi\otimes \chi\eta, \mathbb{C})\not=0$, then $\Pi$ is $(H^\prime, H^{\prime\prime})$-elliptic.
\label{thm-ellipticity}
\end{theorem}

We now define an involution on $\mathcal{C}_c^\infty(G^\prime)$ as follows. Let $f^\prime\in \mathcal{C}_c^\infty(G^\prime)$. Define
\begin{equation}
    f^{\prime\dagger}(g)=f^{\prime}({}^t g^{-1})(\chi\chi^c)(g), \, \, g\in G^\prime.
\label{eq-involution-Gprime}
\end{equation}
A key observation of the relation between the spherical character $I_\Pi$ and the involution defined in \eqref{eq-involution-Gprime} is the following.

\begin{lemma}\cite[Lemma 3.5]{XueZhang2022}
\label{SphericalCharacters-split-lemma1}
For any $f^\prime\in \mathcal{C}_c^\infty(G^\prime)$, we have
\begin{equation*}
    I_\Pi(f^{\prime\dagger})=\epsilon(\Pi\otimes\chi)\chi(-1)^n I_\Pi(f^\prime).
\end{equation*}
\end{lemma}

We also have the following relation between the matching of test functions and the involution defined in \eqref{eq-involution-Gprime}. Let $\varepsilon_D=\eta(\varepsilon)=\pm 1$.

\begin{lemma}\cite[Lemma 3.6]{XueZhang2022}
Let $f^\prime\in \mathcal{C}_c^\infty(G^\prime(F))_0$ and $f\in \mathcal{C}_c^\infty(G(F))_0$ be matching test functions. Then $f^{\prime\dagger}$ and $\eta(-1)^n\varepsilon_{D}^n f$ match. 
\label{lemma-involution-matching}
\end{lemma}

\subsection{The nonsplit side}

Let $\pi$ be an irreducible admissible unitary representation of $G$. We say that $\pi$ is $(H,\chi^{-1})$-distinguished if
\begin{equation*}
    \Hom_H (\pi, \chi^{-1})\not=0,
\end{equation*}
or equivalently, 
\begin{equation*}
    \Hom_H (\pi\otimes\chi, \C)\not=0.
\end{equation*}
In this case, this Hom space is one-dimensional \cite[Theorem 5.1]{Lu2022}. We fix a non-zero element $\ell\in \Hom_H (\pi\otimes \chi , \mathbb{C})$. For $f\in \mathcal{C}_c^\infty(G)$, we define the spherical character attached to $\pi$ by
\begin{equation*}
    J_\pi(f)=\sum_{\phi}\ell(\pi(f)\phi)\overline{\ell(\phi)}
\end{equation*}
where the sum ranges over an orthonormal basis of $\pi$. We note that $\pi$ is $(H,\chi^{-1})$-distinguished if and only if $\pi\otimes\eta$ is $(H,\chi^{-1})$-distinguished, because $\eta$ is trivial on $H$. 

A test function of the form $f=f_1 * \overline{f_1^\vee}$, where $f_1^\vee(g)=f_1(g^{-1})$, is called of positive type. 
\begin{lemma}\cite[Lemma 3.3]{XueZhang2022}
    There is a positive type test function $f\in \mathcal{C}_c^\infty(G(F))_0$ such that $J_\pi(f)>0$.
\label{lemma-positive-type-test-function}
\end{lemma}

By \cite{Guo1998}, the spherical character $J_\pi$ is represented by a locally integrable function $\Theta_\pi$ on $G(F)$, which is locally constant on the regular semisimple locus, and satisfies the equivariant property $\Theta_\pi(h_1 gh_2)=\chi(h_1 h_2)\Theta_\Pi(g)$. We say that $\pi$ is $H$-elliptic if $\Theta_\pi(y)\not=0$ for some elliptic regular semisimple element $y\in G(F)$ and $y$ matches some  $x\in G^\prime(F)$. Similar to Theorem~\ref{thm-ellipticity}, we have the ellipticity of $\pi$.

\begin{proposition}\cite[Proposition 3.4]{XueZhang2022}
If $\pi$ is supercuspidal and $\Hom_H(\pi\otimes\chi, \mathbb{C})\not=0$, then $\pi$ is $H$-elliptic. 
\label{proposition-ellipticity-pi}
\end{proposition}

\subsection{Global arguments}

In this subsection we recall the relative trace formula in \cite{XueZhang2022} and use it to prove certain results. We recall that $A=\Mat_n(D)$, $G=\GL_n(D)$, $H=\GL_n(E)$.
We will use an underline to denote a global object.  We have the following globalization results. 

\begin{lemma}
\label{lemma-globalization-CSA}
There exist the following data:
\begin{enumerate}
    \item A quadratic extension of number fields $\ul E/\ul F$ that splits at all archimedean places, a set of inert finite places $S$ with $|S|=n$, and a non-archimedean inert place $v_0$ of $\ul F$ such that $\ul E_{v_0}/\ul F_{v_0}$ is isomorphic to $E/F$. We denote by $\ul \eta$  the quadratic character attached to $\ul E/\ul F$ by global class field theory.
    
    \item A CSA $\ul A$ over $\ul F$ with an embedding $\ul E\to \ul A$ whose centralizer is $\ul B$, with the property that $(\ul{A}_{v_0}, \ul{B}_{v_0})$ is isomorphic to $(A, B)$, the invariant of $\ul{A}_v$ is $\frac{1}{2n}$ if $v\in S$, and $\ul{A}_v\cong \Mat_{2n}(\ul F_v)$ for all $v\not\in S\cup \{v_0\}$.
    
    \item A character $\ul\chi :\ul E^\times\backslash \mathbb{A}_{\ul E}^\times\to \mathbb{C}^\times$ such that $\ul\chi_{v_0}=\chi$, and $\ul\chi_v$ is unramified at finite places outside $S\cup\{v_0\}$.
\end{enumerate}
\label{SphericalCharacters-Lemma-Xue-3.4}
\end{lemma}

\begin{proof}
The existence of items (1) and (2) can be found in \cite[Lemma 3.4]{Xue2021}. The existence of item (3) is clear.
\end{proof}

With the $\ul{A}$ and $\ul{B}$ as in Lemma \ref{SphericalCharacters-Lemma-Xue-3.4}, we let $\ul{G}=\ul{A}^\times$ and $\ul{H}=\ul{B}^\times$, both viewed as algebraic groups over $\ul F$. Let $\ul{Z}$ be the center of $\ul{G}$. Now we globalize the representation $\pi$.

\begin{lemma}
With the $\ul{A}$, $\ul{B}$ and $\ul\chi$ as in Lemma \ref{SphericalCharacters-Lemma-Xue-3.4}, we can find an irreducible cuspidal automorphic representation $\ul\pi$ of $\ul{G}(\mathbb{A}_{\ul F})$ with central character $\ul \omega$, such that the period integral
\begin{equation*}
    \int_{\ul{H}(\ul F)\ul{Z}(\mathbb{A}_{\ul F})\backslash \ul{H}(\mathbb{A}_{\ul F})} \varphi(h)\ul \chi(h)dh
\end{equation*}
is not identically zero, where ${\ul \pi}_{v_0}\cong \pi$, 
${\ul \pi}_v$ is a ($\ul{H}({\ul F}_v),
{\ul\chi}_v^{-1})$-distinguished representation for $v\in S$,
${\ul\pi}_w$ is a supercuspidal ($\ul{H}({\ul F}_w),
{\ul\chi}_w^{-1})$-distinguished representation for some split place $w$ of $\ul F$, 
and ${\ul \pi}_u$ is unramified at all other non-archimedean places $u$.
\label{SphericalCharacters-Lemma-Prasad-Schulze-Pillot}
\end{lemma}

\begin{proof}
This follows from \cite[Theorem 4.1]{PrasadSchulze-Pillot2008}. We note that $\ul H/\ul Z$ has no $\ul F$-rational characters. 
\end{proof}

Let us now recall the relative trace formula developed in \cite{XueZhang2022}. We start with the non-split side. Let $\sigma$ be an irreducible cuspidal automorphic representation of $\ul{G}({\mathbb{A}_{\ul F}})$. For $\varphi\in V_\sigma$, we define a global period
\begin{equation*}
    P_{{\ul\chi}}(\varphi)=\int_{{\ul Z}({\mathbb{A}_{\ul F}}) {\ul H}(\ul F) \backslash {\ul H}({\mathbb{A}_{\ul F}})} \varphi(h){\ul\chi}(h)dh.
\end{equation*}
We say that $\sigma$ is globally $({\ul H}({\mathbb{A}_{\ul F}}),{\ul \chi}^{-1})$-distinguished if ${P_{\ul \chi}}(\varphi)$ is not identically zero. We define a global distribution
\begin{equation}
    J_\sigma(\mathbf{f})=\sum_{\varphi} {P_{\ul \chi}}(\pi(\mathbf{f})\varphi)\overline{{P_{\ul \chi}}(\varphi)}, \quad \mathbf{f}\in \mathcal{C}_c^\infty({\ul G}({\mathbb{A}_{\ul F}})),
\end{equation}
where $\varphi$ runs through an orthonormal basis of $\sigma$. Then $\sigma$ is globally $({\ul H}({\mathbb{A}_{\ul F}}),{\ul \chi}^{-1})$-distinguished if and only if $J_\sigma\not=0$.

Now we consider the split side. We denote by ${\ul G}^\prime=\Res_{{\ul E/\ul F}} \GL_{2n}$, ${\ul H}^\prime=\Res_{{\ul E/\ul F}}(\GL_{n}\times\GL_{n})$, and ${\ul H}^{\prime\prime}=\GL_{2n}$. Let ${\ul Z}^\prime$ be the center of ${\ul G}^\prime$ and ${\ul Z_{2n}}$ be the center of ${\ul H}^{\prime\prime}$. Let $\sigma^\prime$ be an irreducible automorphic representation of ${\ul G}^\prime({\mathbb{A}_{\ul F}})$. For $\varphi\in V_{\sigma^\prime}$, we define the global periods
\begin{equation*}
    P^\prime_{\ul \chi}(\varphi)=\int_{{\ul Z}^\prime({\mathbb{A}_{\ul F}}) {\ul H}^\prime(\ul F) \backslash {\ul H}^\prime({\mathbb{A}_{\ul F}})} \varphi\left( \begin{pmatrix} h_1 & \\ & h_2\end{pmatrix}\right) {\ul \chi}(h_1 \overline{h_2})dh_1 dh_2,
\end{equation*}
and
\begin{equation*}
    P^{\prime\prime}_{{\ul \chi}{\ul \eta}}(\varphi)=\int_{\ul Z_{2n}({\mathbb{A}_{\ul F}}) {\ul H}^{\prime\prime}(\ul F) \backslash {\ul H}^{\prime\prime}({\mathbb{A}_{\ul F}})} \varphi(h)({\ul \chi}{\ul \eta})(h)dh.
\end{equation*}
Define the character ${\ul \chi}_{{\ul H}^\prime}$ on ${\ul H}^\prime({\mathbb{A}_{\ul F}})$ by ${\ul \chi}_{{\ul H}^\prime}(h^\prime)={\ul \chi}(h_1\overline{h_2})$ for $h^\prime=(h_1, h_2)\in {\ul H}^\prime({\mathbb{A}_{\ul F}})$.
We say that $\sigma^\prime$ is globally $({\ul H}^\prime({\mathbb{A}_{\ul F}}),{\ul \chi}_{{\ul H}^\prime}^{-1})$-distinguished if $P^\prime_{\ul \chi}$ is not identically zero. The linear period $ P^\prime_{\ul \chi}$ is not identically zero if and only if $L(\frac{1}{2}, \sigma^\prime\otimes{\ul \chi})\not=0$ and $L(s, \sigma^\prime, \wedge^2\otimes {\ul \chi}{\ul \chi}^c)$ has a pole at $s=1$. The later condition implies that ${\sigma^\prime}^\vee\cong \sigma^\prime\otimes {\ul \chi}{\ul \chi}^c$. 
Similarly, we say that $\sigma^\prime$ is globally $({\ul H}^{\prime\prime}({\mathbb{A}_{\ul F}}),({\ul \chi}{\ul \eta})^{-1})$-distinguished if $P^{\prime\prime}_{{\ul \chi}{\ul \eta}}$ is not identically zero. 
The linear period $P^{\prime\prime}_{{\ul \chi}{\ul \eta}}$ is not identically zero if and only if the Asai $L$-function $L(s, \sigma^\prime\otimes{\ul \chi}, \Asai^-)$ has a pole at $s=1$. This condition implies that ${\sigma^\prime}^\vee \otimes{\ul \chi}^{-1}\cong {\sigma^\prime}^c\otimes{\ul \chi}^c$. We refer the reader to \cite[\S 3.1]{XueZhang2022} for more details.

We recall the following result.

\begin{lemma}\cite[Lemma 3.2]{XueZhang2022}
Let $\sigma^\prime$ be an irreducible automorphic cuspidal representation of ${\ul G}^\prime({\mathbb{A}_{\ul F}})$. If neither of $P^\prime_{\ul \chi}$ and $P^{\prime\prime}_{{\ul \chi}{\ul \eta}}$ is identically zero, then there is an irreducible cuspidal automorphic representation $\tau$ of ${\ul H}^{\prime\prime}({\mathbb{A}_{\ul F}})$ such that $\sigma^\prime=\BC(\tau)$. Moreover, $L(\frac{1}{2}, \sigma^\prime\otimes{\ul \chi})\not=0$, and $L(s, \tau, \wedge^2\otimes{\ul \chi}|_{{\mathbb{A}_{\ul F}}^\times})$ has a simple pole at $s=1$.
\label{lemma-XueZhang-3.2}
\end{lemma}

The following theorem is due to the works in \cite{JacquetShalika1990, AsgariShahidi2006, AsgariShahidi2014, HundleySayag2016}.

\begin{theorem}
\label{theorem-functoriality-GSpin}
Let $\tau$ be an irreducible cuspidal automorphic representation of $\GL_{2n}(\A_{\ul F})$ with central character $\omega_\tau$. Let $\psi:\ul F\backslash   \A_{\ul F}\to \mathbb{C}^\times$ be a non-trivial additive character and $\xi:\ul F^\times\backslash \A_{\ul F}^\times \to \mathbb{C}^\times$ a multiplicative character such that $\xi^n=\omega_\tau$. The following are equivalent:
\begin{enumerate}
\item[(i)] There is a $\varphi\in V_\tau$ and $g\in \GL_{2n}(\A_{\ul F})$ such that the global Shalika period
\begin{equation*}
\int_{[\Mat_n]}\int_{ \A_{\ul F}^\times \GL_n(\ul F)\backslash \GL_n(\A_{\ul F})}	(\tau(g)\varphi)\left( \begin{pmatrix} h& \\&h \end{pmatrix} \begin{pmatrix} 1_n&X\\ &1_n\end{pmatrix} \right) \xi^{-1}(\det(h))\psi^{-1}(\tr(X))dhdX\not=0.
\end{equation*}

\item[(ii)] Let $S$ be a finite set of places outside including the archimedean ones of which $\tau$, $\psi$, $\xi$ are unramified. The twisted partial exterior square $L$-function
\begin{equation*}
L^S(s,\tau, \wedge^2\otimes\xi^{-1}):=\prod_{v\not\in S}L(s, \tau_v, \wedge^2\otimes\xi_v^{-1})	
\end{equation*}
has a pole at $s=1$.

\item[(iii)] $\tau$ is the transfer of a globally generic cuspidal automorphic representation of $\GSpin_{2n+1}(\A_{\ul F})$ whose central character is equal to $\xi$.
\end{enumerate}
\end{theorem}

For $\mathbf{f}^\prime\in \mathcal{C}_c^\infty({\ul G}^\prime({\mathbb{A}_{\ul F}}))$, put
\begin{equation*}
    I_{\sigma^\prime}(\mathbf{f}^\prime)=\sum_{\varphi} P^\prime_{\ul \chi}( \sigma^\prime(\mathbf{f}^\prime) \varphi) \overline{P^{\prime\prime}_{{\ul \chi}{\ul \eta}}(\varphi)},
\end{equation*}
where $\varphi$ sums over an orthonormal basis of $\sigma^\prime$. This is a global spherical character attached to $\sigma^\prime$.

Given two global test functions $\mathbf{f}=\otimes f_v\in \mathcal{C}_c^\infty({\ul G}({\mathbb{A}_{\ul F}}))$ and $\mathbf{f}^\prime=\otimes f_v^\prime\in \mathcal{C}_c^\infty({\ul G}^\prime({\mathbb{A}_{\ul F}}))$, we say that $\mathbf{f}$ and $\mathbf{f}^\prime$ match if $f_v\in \mathcal{C}_c^\infty({\ul G}({\ul F_v}))_0$ and $f_v^\prime\in \mathcal{C}_c^\infty({\ul G}^\prime({\ul F_v}))_0$ and they match for all places $v$ of $\ul F$. 

\begin{proposition}
Let $\sigma$ be an irreducible cuspidal automorphic representation of ${\ul G}(\mathbb{A}_{\ul F})$ with central character $\ul \omega$ such that $\sigma_{v_0}\cong \pi$, $\sigma_v$ is a (${\ul H}({\ul F_v}),
{\ul \chi}_v^{-1})$-distinguished representation of ${\ul G}({\ul F_v})$ if $v\in S$, and $\sigma_w$ is a supercuspidal (${\ul H}(\ul F_w), {\ul \chi}_w^{-1})$-distinguished representation for some split place $w$ of $K$. Let $\sigma^\prime=\JL(\sigma)$ be the Jacquet-Langlands transfer of $\sigma$ to ${\ul H}^{\prime\prime}({\mathbb{A}_{\ul F}})$, and $\BC(\sigma^\prime)$ the base change of $\sigma^\prime$ to ${\ul G}^\prime({\mathbb{A}_{\ul F}})$. Suppose that $\mathbf{f}=\otimes f_v\in \mathcal{C}_c^\infty({\ul G}({\mathbb{A}_{\ul F}}))$ and $\mathbf{f}^\prime=\otimes f_v^\prime\in \mathcal{C}_c^\infty({\ul G}^\prime({\mathbb{A}_{\ul F}}))$ match. Assume that if $v\in S$, $f_{v}$ is a positive type test function in $\mathcal{C}_c^\infty({\ul G}({\ul F_v}))_0$ supported sufficiently close to 1 such that $J_{\sigma_v}(f_v)\not=0$, that $f_{v_0}$ is supported in the elliptic locus, and that $f_w=f_w^\prime$ is an essential matrix coefficient of $\sigma_w$. Then we have
\begin{equation}
    I_{\BC(\sigma^\prime)}(\mathbf{f}^\prime) = J_\sigma(\mathbf{f})+J_{\sigma\otimes\eta}(\mathbf{f}).
\label{eq-SphericalCharacters-trace-formula}
\end{equation}
\label{prop-SphericalCharacters-trace-formula}
\end{proposition}

The proof essentially follows from \cite[\S 4.1]{XueZhang2022}. For completeness, we reproduce the proof here, after we introduce some necessary notations and tools from \cite{Beuzart-PlessisLiuZhangZhu2021}.

\begin{proof}
Let $\mathcal{A}$ be a complex algebra. Recall that a multiplier is a complex linear map $\mu\star:\mathcal{A}\to \mathcal{A}$ that commutes with the left and right multiplications in $\mathcal{A}$. The space of multipliers of $\mathcal{A}$ is denoted by $\mathrm{Mul}(\mathcal{A})$. 

Let $v$ be an archimedean place of $\ul F$, $\mathcal{S}(\ul G(\ul F_v))$ the Schwartz space of $\ul G(\ul F_v)$,  $\mathfrak{t}_v$ the complexified Cartan subalgebra of $\ul G(\ul F_v)$, $\mathfrak{t}_v^*$ the dual space of $\mathfrak{t}_v$, and $\mathcal{Z}_{\ul G(\ul F_v)}\cong \C[\mathfrak{t}_v]^{W_v}$ the center of the universal enveloping algebra of $\ul G(\ul F_v)$. We write $\ul \chi_v=(\chi_1, \chi_2)$. Then the character $\chi_1\chi_2$ of $\ul F_v^\times$ defines an element $\ul a_v \in \mathfrak{t}_v^*$.

Let $\mathcal{M}_v$ be the space of holomorphic functions on $\mathfrak{t}_v^*$ defined in \cite[Definition 2.8 (3)]{Beuzart-PlessisLiuZhangZhu2021} (note that the notation is $\mathcal{M}_\theta^\sharp(\mathfrak{h}_\C^*)$ in \cite{Beuzart-PlessisLiuZhangZhu2021}). We refer the reader to \cite[Theorem 2.13]{Beuzart-PlessisLiuZhangZhu2021} for the following property. There is an algebra homomorphism
\begin{equation*}
\mathcal{M}_v \to \mathrm{Mul}(\mathcal{S}(\ul G(\ul F_v)))	, \quad \mu\mapsto \mu\star
\end{equation*}
such that
\begin{equation*}
\sigma(\mu\star f) = \mu(\lambda_\sigma)\sigma(f), 	
\end{equation*}
for every $f\in \mathcal{S}(\ul G(\ul F_v))$
and every irreducible admissible representation $\sigma$  of $\ul G(\ul F_v)$,
where $\lambda_\sigma$ is the infinitesimal character of $\sigma$. 
Let
$\iota_v$ be the involution on $\mathcal{M}_v$ such that $\iota_v(\mu)(z) = \mu(-
\underline{a}_v - z)$ for all $z \in \mathfrak{t}_v^*$. Let $\mathcal{M}_v^+$ be the subspace of $\mathcal{M}_v$
consisting of elements invariant under $\iota_v$. Note that the map $f\mapsto f^\vee \chi_1\chi_2$, where $f^\vee(g)=f(g^{-1})$, is an involution on $\mathcal{S}(\ul G(\ul F_v))$. Define $\mathcal{S}(\ul G(\ul F_v))^+=\{f\in \mathcal{S}(\ul G(\ul F_v)): f^\vee \chi_1\chi_2=f\}$. Then for $f\in \mathcal{S}(\ul G(\ul F_v))^+$ and $\mu\in \mathcal{M}_v^+$, we have $\mu\star f\in \mathcal{S}(\ul G(\ul F_v))^+$.
We denote $\mathcal{Z}_{\ul G}=\prod_{v|\infty} \mathcal{Z}_{\ul G(\ul F_v)}$, $\lambda_\infty=\otimes_{v|\infty}\lambda_v$, $\mathcal{M}=\prod_{v|\infty}\mathcal{M}_v$, and $\mathcal{M}^+=\prod_{v|\infty}\mathcal{M}_v^+$.

We let $\mathtt{S}$ be a finite set of finite places of $\ul F$ such that if $v\not\in \mathtt{S}$, then $\ul E_v/\ul F_v$, $\sigma_v$ and $\ul\chi_v$ are all unramified. Recall that $S$ is a finite set of inert finite places as in Lemma~\ref{lemma-globalization-CSA}. Denote $\mathtt{T}=\mathtt{S}\cup S$ and let
\begin{equation*}
\mathcal{H}_{\ul G}^{\mathtt{T}}	=\otimes_{v\nmid\infty, v\not\in \mathtt{T}} \mathcal{H}_v= \otimes_{v\nmid\infty, v\not\in \mathtt{T}} \mathcal{C}_c^\infty(\ul G(\mathfrak{o}_{\ul F_v}) \backslash \ul G(\ul F_v)/ \ul G(\mathfrak{o}_{\ul F_v}) ) 
\end{equation*}
be the spherical Hecke algebra away from $\mathtt{T}$, where $\mathfrak{o}_{\ul F_v}$ is the ring of integers of $\ul F_v$ at a finite place $v\not\in \mathtt{T}$. Let $K=\prod_{v\nmid \infty}K_v$ be a fixed compact open subgroup such that $K_v=\ul G(\mathfrak{o}_{\ul F_v})$ if $v\not\in \mathtt{S}$. We denote by $\mathcal{C}_c^\infty(\ul G(\A_{\ul F}))_K$ the subalgebra of $\mathcal{C}_c^\infty(\ul G(\A_{\ul F}))$ of bi-$K$-invariant functions. 

The above objects also have their counterparts for $\ul G^\prime$. We keep $v$ an archimedean place of $\ul F$, and we have an algebra of homomorphic functions $\mathcal{M}_v^\prime$, which is identified with $\mathcal{M}_v\otimes\mathcal{M}_v$. Put $\mathcal{M}_v^{\prime+}=\mathcal{M}_v^{+}\otimes \mathcal{M}_v^{+}$, $\mathcal{M}^{\prime}=\prod_{v|\infty} \mathcal{M}_v^{\prime}$, $\mathcal{M}^{\prime+}=\prod_{v|\infty} \mathcal{M}_v^{\prime+}$. We have $\mathcal{S}(\ul G^\prime(\ul F_v))^+=\mathcal{S}(\ul G(\ul F_v))^+\otimes \mathcal{S}(\ul G(\ul F_v))^+$. The universal enveloping algebra $\mathcal{Z}_{\ul G^\prime}$ is identified with $\mathcal{Z}_{\ul G}\otimes \mathcal{Z}_{\ul G}$, and the spherical Hecke algebra away from $\mathtt{T}$ is
\begin{equation*}
\mathcal{H}_{\ul G^\prime}^{\mathtt{T}}	=\otimes_{v\nmid\infty, v\not\in \mathtt{T}} \mathcal{H}_{\ul G^\prime, v}= \mathcal{H}_{\ul G}^{\mathtt{T}}\otimes \mathcal{H}_{\ul G}^{\mathtt{T}}.
\end{equation*}
We have a base change homomorphism $\mathrm{bc}:\mathcal{Z}_{\ul G^\prime}\otimes \mathcal{H}_{\ul G^\prime}^{\mathtt{T}}	\to \mathcal{Z}_{\ul G}\otimes \mathcal{H}_{\ul G}^{\mathtt{T}} $ given by the usual multiplication in $\mathcal{Z}_{\ul G}$ and $\mathcal{H}_{\ul G}^{\mathtt{T}}$. Let $K^\prime=\prod_{v\nmid \infty}K_v^\prime$ be a fixed compact open subgroup such that $K_v^\prime=\ul G^\prime(\mathfrak{o}_{\ul F_v})$ if $v\not\in \mathtt{S}$. Let $\mathcal{C}_c^\infty(\ul G^\prime(\A_{\ul F}))_{K^\prime}$ be the subalgebra of $\mathcal{C}_c^\infty(\ul G^\prime(\A_{\ul F}))$ of bi-$K^\prime$-invariant functions. 

Let $\lambda=(\lambda_\infty, \lambda^{\infty, \mathtt{T}})$ be the character of $\mathcal{Z}_{\ul G}\otimes\mathcal{H}_{\ul G}^{\mathtt{T}}$ associated to $\sigma$, and let $\lambda^\prime=\lambda\circ \mathrm{bc}=(\lambda, \lambda)$. Then $\lambda^\prime$ is the character of $\mathcal{Z}_{\ul G^\prime}\otimes\mathcal{H}_{\ul G^\prime}^{\mathtt{T}}$ associated to $\BC(\sigma^\prime)$ (see \cite[Chapter 1, Section 5]{ArthurClozel1989}). Let $L_0^2(\ul G(\ul F)\backslash \ul G(\A_{\ul F})/K, \ul \omega)[\lambda]$ and $L_0^2(\ul G^\prime(\ul F)\backslash \ul G^\prime(\A_{\ul F})/K^\prime, \ul \omega^\prime)[\lambda^\prime]$ be the maximal quotients of the spaces $L_0^2(\ul G(\ul F)\backslash \ul G(\A_{\ul F})/K, \ul \omega)$  and $L_0^2(\ul G^\prime(\ul F)\backslash \ul G^\prime(\A_{\ul F})/K^\prime, \ul \omega)$ on which $\mathcal{Z}_{\ul G}\otimes\mathcal{H}_{\ul G}^{\mathtt{T}}$ and $\mathcal{Z}_{\ul G^\prime}\otimes\mathcal{H}_{\ul G^\prime}^{\mathtt{T}}$ acts by $\lambda$ and $\lambda^\prime$ respectively. Then we have
\begin{equation*}
L_0^2(\ul G(\ul F)\backslash \ul G(\A_{\ul F})/K, \ul \omega)[\lambda]=\sigma\oplus (\sigma\otimes\ul \eta), \quad L_0^2(\ul G^\prime(\ul F)\backslash \ul G^\prime(\A_{\ul F})/K^\prime, \ul \omega^\prime)[\lambda^\prime]=\BC(\sigma^\prime).	
\end{equation*}
Let 
\begin{equation*}
\mathrm{bc}: \mathcal{M}^\prime\otimes \mathcal{H}_{\ul G^\prime}^{\mathtt{T}}= (\mathcal{M} \otimes \mathcal{H}_{\ul G}^{\mathtt{T}})	\otimes (\mathcal{M} \otimes \mathcal{H}_{\ul G}^{\mathtt{T}})\to \mathcal{M} \otimes \mathcal{H}_{\ul G}^{\mathtt{T}}
\end{equation*}
be the map given by multiplication. 

By \cite[Proposition 3.7]{XueZhang2022}, there exist elements $\mu^\prime\in \mathcal{M}^{\prime+}\otimes \mathcal{H}_{\ul G^\prime}^{\mathtt{T}}$ and $\mu=\mathrm{bc}(\mu^\prime)\in \mathcal{M}^{+}\otimes \mathcal{H}_{\ul G}^{\mathtt{T}}$ such that, for all  and $f\in \mathcal{C}_c^\infty(\ul G(\A_{\ul F}))$, we have that
\begin{itemize}
\item $R(\mu^\prime\star f^\prime)$ maps $L^2(\ul G^\prime(\ul F)\backslash \ul G^\prime(\A_{\ul F})/K^\prime, \ul \omega^\prime)$ into $\BC(\sigma^\prime)$ for all $f^\prime\in \mathcal{C}_c^\infty(\ul G^\prime(\A_{\ul F}))$ ,
\item $\BC(\sigma^\prime)(\mu^\prime\star f^\prime)=\BC(\sigma^\prime)(f^\prime)$ for all $f^\prime\in \mathcal{C}_c^\infty(\ul G^\prime(\A_{\ul F}))_{K^\prime}$,
\end{itemize}
and
\begin{itemize}
\item $R(\mu\star f)$ maps $L^2(\ul G(\ul F)\backslash \ul G(\A_{\ul F})/K, \ul \omega)$ into $\sigma\oplus (\sigma\otimes\ul \eta)$ for all $f\in \mathcal{C}_c^\infty(\ul G(\A_{\ul F}))$ ,
\item $\sigma(\mu\star f)=\sigma(f)$ for all $f\in \mathcal{C}_c^\infty(\ul G(\A_{\ul F}))_{K}$.
\end{itemize}
We emphasize that the multipliers $\mu^\prime$ and $\mu$ are in the ``plus" subspaces. 
Let $\mathbf{f}=\otimes f_v\in \mathcal{C}_c^\infty({\ul G}({\mathbb{A}_{\ul F}}))$ and $\mathbf{f}^\prime=\otimes f_v^\prime\in \mathcal{C}_c^\infty({\ul G}^\prime({\mathbb{A}_{\ul F}}))$ be the matching test functions given in Proposition~\ref{prop-SphericalCharacters-trace-formula}. 
We now use the test functions $\mu^\prime\star \mathbf{f}^\prime$ and $\mu\star \mathbf{f}$, which still match. We conclude that
\begin{equation*}
    I_{\BC(\sigma^\prime)}(\mathbf{f}^\prime) = J_\sigma(\mathbf{f})+J_{\sigma\otimes\eta}(\mathbf{f}).
\end{equation*}	
\end{proof}

As an application of Proposition~\ref{prop-SphericalCharacters-trace-formula}, we have the following result relating a $(H, \chi^{-1})$-distinguished supercuspidal representation of $G$ to a simultaneously $(H^\prime, \chi_{H^\prime}^{-1})$-distinguished and $(H^{\prime\prime}, \chi^{-1} \eta )$-distinguished representation of $G^\prime$ through the Jacquet-Langlands transfer. 

\begin{theorem}
Let $\pi$ be an irreducible $(H, \chi^{-1})$-distinguished supercuspidal representation of $G$ with central character $\omega$. Let $\pi_{0,E}$ be the base change of $\pi_0$. Then $\pi_{0,E}$ is both $(H^\prime, \chi_{H^\prime}^{-1})$-distinguished and $(H^{\prime\prime}, \chi^{-1}\eta  )$-distinguished.
\label{theorem-forward-direction-distinghished-1}
\end{theorem}

\begin{proof}
Since $\pi$ is $(H, \chi^{-1})$-distinguished, by Lemma~\ref{lemma-positive-type-test-function}, there exists some positive type test function $f\in \mathcal{C}_c^\infty(G)$ such that $J_\pi(f)>0$.
By Lemma~\ref{SphericalCharacters-Lemma-Prasad-Schulze-Pillot}, we can find a globally $({\ul H}({\mathbb{A}_{\ul F}}), {\ul \chi}^{-1})$-distinguished cuspidal representation $\ul \pi$ of ${\ul G}({\mathbb{A}_{\ul F}})$ and a test function $\mathbf{f}=\otimes f_v\in \mathcal{C}_c^\infty(\ul G({\mathbb{A}_{\ul F}}))$ with $\mathbf{f}_{v_0}=f$ so that they satisfy the conditions of Proposition~\ref{prop-SphericalCharacters-trace-formula}. Let $\mathbf{f}^\prime=\otimes f_v^\prime\in \mathcal{C}_c^\infty({\ul G}^\prime({\mathbb{A}_{\ul F}}))$ be a function that matches $\mathbf{f}$. Since the test function $f$ is of positive type, by Proposition~\ref{prop-SphericalCharacters-trace-formula}, we have $I_{\BC(\sigma^\prime)}(\mathbf{f}^\prime)>0$, where $\sigma^\prime=\JL(\ul \pi)$. Since $\pi_{0, E}$ is the local component of $\BC(\sigma^\prime)$, we conclude that $\pi_{0,E}$ is both $(H^\prime, \chi_{H^\prime}^{-1})$-distinguished and $(H^{\prime\prime}, \chi^{-1}\eta)$-distinguished. 
\end{proof}

\section{The forward direction}
\label{section-forward}

In this section, following a similar argument as in \cite[Section 4]{Xue2021}, we prove Theorem~\ref{thm-main-forward}, which is restated as follows. Recall that we fix $\varepsilon\in NE^\times$ (resp. $F^\times\backslash NE^\times$) if $D$ splits (resp. ramifies) and the group $G$ is realized as a subgroup of $\GL_{2n}(E)$ which consists of elements of the form $\begin{bmatrix} \alpha &\varepsilon \beta\\ \overline{\beta}& \overline{\alpha}\end{bmatrix}$, $\alpha, \beta\in \GL_n(E)$. Let $\varepsilon_D=1$ (resp. $\varepsilon_D=-1$) if $D$ splits (resp. ramifies).

\begin{theorem}
Let $\pi$ be an irreducible $(H, \chi^{-1})$-distinguished representation of $G$ such that its Jacquet-Langlands transfer $\pi_0=\JL(\pi)$ is generic with central character $\omega$. 
Then the following two conditions hold:
\begin{itemize}
    \item[(1)] The Langlands parameter of $\pi_0$ takes values in $\GSp_{2n}(\C)$ with similitude factor $\chi^{-1}|_{F^\times}$.
    
    \item[(2)] $\varepsilon(\pi_{0,E}\otimes\chi)=\varepsilon_{D}^n \eta(-1)^n \chi(-1)^n$.
\end{itemize}
\label{theorem-forward-direction-main}
\end{theorem}

\subsection{The supercuspidal case}
\label{subsection-forward-supercuspidal}
We first prove Theorem~\ref{theorem-forward-direction-main} under the assumption that $\pi$ is supercuspidal. 

\begin{proof}[Proof of Theorem~\ref{theorem-forward-direction-main} assuming $\pi$ is supercuspidal]
Assume that $\pi$ be an irreducible $(H, \chi^{-1})$-distinguished supercuspidal representation of $G$. We keep the notations from the proof of Theorem~\ref{theorem-forward-direction-distinghished-1}. In particular, we have matching test functions $\mathbf{f}=\otimes f_v\in \mathcal{C}_c^\infty({\ul G}({\mathbb{A}_{\ul F}}))$ with $f_{v_0}=f$ so that they satisfy the conditions of Proposition~\ref{prop-SphericalCharacters-trace-formula} and $\mathbf{f}^\prime=\otimes f_v^\prime\in \mathcal{C}_c^\infty({\ul G}^\prime({\mathbb{A}_{\ul F}}))$, so that 
\begin{equation*}
    I_{\BC(\sigma^\prime)}(\mathbf{f}^\prime) = J_{\ul \pi}(\mathbf{f})+J_{\ul\pi\otimes\ul\eta}(\mathbf{f})>0.
\end{equation*}
Recall that $\sigma^\prime=\JL(\ul \pi)$ is the Jacquet-Langlands transfer of $\ul \pi$.
By Lemma~\ref{lemma-XueZhang-3.2}, $L(s, \sigma^\prime, \wedge^2\otimes{\ul \chi}|_{{\mathbb{A}_{\ul F}}^\times})$ has a simple pole at $s=1$. Since $\pi_0$ is the local component of $\sigma^\prime=\JL(\ul \pi)$, by Theorem~\ref{theorem-functoriality-GSpin} we conclude that the Langlands parameter of $\pi_0$ takes values in $\GSp_{2n}(\C)$ with similitude factor $\chi^{-1}|_{F^\times}$.

Now we move on to compute the local root number $\varepsilon(\pi_{0,E}\otimes\chi)$. The main idea is to use the involution defined in \eqref{eq-involution-Gprime}. Note that $\varepsilon_{D_v}^n\eta_v(-1)^nf_v$ and $f_v^{\prime\dagger}$ also match, by Lemma~\ref{lemma-involution-matching}. We define a global test function $\mathbf{f}^{\prime\dagger}$ by
\begin{equation*}
    \mathbf{f}^{\prime\dagger}=\otimes_{w\not=v_0}f_w^\prime\otimes f_{v_0}^{\prime\dagger}.
\end{equation*}
Now we use Lemma~\ref{SphericalCharacters-split-lemma1} to conclude that
\begin{equation*}
    \varepsilon(\pi_{0,E}\otimes\chi)\chi(-1)^n I_{\BC(\sigma^\prime)}(\mathbf{f}^\prime)= \varepsilon_{D}^n\eta(-1)^n \left( J_\sigma(\mathbf{f})+J_{\sigma\otimes\eta}(\mathbf{f})\right)\not=0.
\end{equation*}
Thus
\begin{equation*}
     \varepsilon(\pi_{0,E}\otimes\chi)= \varepsilon_{D}^n\eta(-1)^n \chi(-1)^n.
\end{equation*}
This completes the proof of Theorem~\ref{theorem-forward-direction-main} under the assumption that $\pi$ is supercuspidal.  
\end{proof}

\subsection{The general case}

Before we proceed to the general case, we recall some generalities on representations of $\GL_r(C)$, where $C$ be a central division algebra of dimension $d^2$ over $F$. 
Let $\rho_1, \cdots, \rho_s$ be irreducible representations of $\GL_{r_1}(C), \cdots, \GL_{r_s}(C)$ respectively, and we denote by 
\begin{equation*}
\rho_1\times\cdots \cdots \times \rho_s	
\end{equation*}
the normalized induced representation of $\GL_r(C)$, associated to the usual standard upper triangular parabolic subgroup corresponding to the partition $r=r_1+\cdots+r_s$. We denote by $\nu$ the absolute value of the reduced norm of any CSA. Suppose that $r=s\ell$ and $\rho$ is a supercupsidal representation of $\GL_s(C)$. We consider the case $C=F$ first. Then $\rho$ is a supercuspidal representation of $\GL_s(F)$, and the representation
\begin{equation*}
\rho\times\rho\nu \times\cdots \times \rho\nu^{\ell-1}	
\end{equation*}
has a unique irreducible quotient which is a discrete series representation of $\GL_r(F)$. Moreover, any discrete series representation of $\GL_r(F)$ is obtained in this way. In general, assume that $\rho^\prime=\JL(\rho)$ is the Jacquet-Langlands transfer of $\rho$ to $\GL_{sd}(F)$. Then it is an irreducible quotient of $\tau\times \cdots \times \tau\nu^{q-1}$. Set $\nu_\rho=\nu^q$. Again, the normalized  parabolic induction
\begin{equation*}
\rho\times \rho \nu_\rho \times \cdots \times \rho\nu_\rho^{\ell-1}	
\end{equation*}
has a unique quotient which is a discrete series representation of $\GL_r(C)$. Moreover, all discrete series representation of $\GL_r(C)$ arise in this way. We call such a representation a segment, and denote it by $\Delta=\{\rho, \rho\nu_\rho, \cdots, \rho\nu_\rho^{\ell-1} \}$.

Let $\WD_F=W_F\times\SL_2(\C)$ be the Weil-Deligne group of $F$. 
To each irreducible representation $\pi^\prime$ of $\GL_{2n}(F)$, one can associate a Weil-Deligne representation
\begin{equation*}
\phi_{\pi^\prime}:\WD_F\to \GL_{2n}(\C)
\end{equation*}
by the local Langlands correspondence. 
If $\pi^\prime$ is supercuspidal, then $\phi_{\pi^\prime}$ is an irreducible representation of $W_F$ and is trivial on $\SL_2(\C)$. If $\pi^\prime$ is a segment of the form
\begin{equation*}
\{\tau, \cdots, \tau\nu^{\ell-1}\},	
\end{equation*}
then $\phi_{\pi^\prime}=\phi_{\tau}\boxtimes \Sym^{\ell-1}$, where $\phi_\tau$ is an irreducible representation of $W_F$ associated to $\tau$, and $\Sym^{\ell-1}$ is the unique irreducible algebraic $\ell$-dimensional representation of $\SL_2(\C)$. Moreover, the local root number of $\pi^\prime$ is given by
\begin{equation*}
\varepsilon(\pi^\prime)=\varepsilon(\phi_{\pi^\prime})=\varepsilon(\phi_\tau)^{\ell} \det(-\mathrm{Frob}| \phi_\tau^{I_F})^{\ell-1},
\end{equation*}
where $I_F$ is the inertia subgroup of $W_F$, and $\phi_\tau^{I_F}$ stands for the subspace of $\phi_\tau$ on which $I_F$ acts trivially. We note that if $\phi_\tau$ is not one-dimensional, then $\phi_\tau^{I_F}=0$.

Using the classification of $(H,\chi^{-1})$-distinguished representations, the proof of Theorem~\ref{theorem-forward-direction-main} reduces to the case of discrete series. Assume that $\pi$ is $(H,\chi^{-1})$-distinguished. By  \cite[Proposition 3.4]{Suzuki2023}, $\pi$ is a quotient of $\Delta_1\times\cdots \times\Delta_s$ where each $\Delta_i$ is an irreducible discrete series representation of $\GL_{n_i}(D)$ with $n_1+\cdots+n_s=n$, and there is an involutive permutation $\zeta\in \mathfrak{S}_s$ such that 
\begin{itemize}
	\item $n_{\zeta(i)}=n_i$ for each $i$. 
	\item if $\zeta(i)=i$, then $n_i$ is even and $\Delta_i$ is $(H_i, \chi^{-1})$-distinguished. Here $H_i$ denotes the centralizer of $E^\times$ in $\GL_{n_i}(D)$.
	\item if $\zeta(s)\not=i$, $\Delta_{\zeta(i)}\cong \Delta_i^\vee\cdot \chi^{-1}$.
\end{itemize}
Using this description, we have the following result.

\begin{lemma}\cite[Corollary 3.5]{Suzuki2023}
	Theorem~\ref{theorem-forward-direction-main} for discrete series implies Theorem~\ref{theorem-forward-direction-main}.
\end{lemma}

Now assume that $\pi$ is a discrete series representation. We write $\pi$ as a segment
\begin{equation*}
\{\rho\nu_\rho^{-(\ell-1)/2}, \cdots, \rho\nu_\rho^{(\ell-1)/2}\}
\end{equation*}
where $\rho$ is an irreducible supercuspidal representation of $\GL_s(D)$, $s\ell=n$, $\nu_\rho=\nu^q$, and $\nu$ is the absolute value of the reduced norm. 
The case $s=1$ and $\rho$ being one-dimensional is proved in \cite{Chommaux2019}, hence from now on we assume $\rho$ is not one-dimensional.

First we assume $\ell$ is even. Note that this implies $n$ is even and hence $\varepsilon_{D}^n \eta(-1)^n \chi(-1)^n=1$. Then by the proof of \cite[Proposition 5.6]{BroussousMatringe2021}, we have $\rho\cong \chi^{-1}|_{F^\times}\otimes \rho^\vee$. Write $\pi_0=\JL(\pi)$ as a segment
\begin{equation*}
\{\tau\nu^{-(\ell^\prime-1)/2}, \cdots, \tau\nu^{(\ell^\prime-1)/2}\},
\end{equation*}
where $\tau$ is a supercuspidal representation of $\GL_{2n/{\ell^\prime}}(F)$, $\ell^\prime$ is even, and $\tau\cong \chi^{-1}|_{F^\times}\otimes\tau^\vee$. Let $\phi_\tau$ be the representation of the Weil group of $F$ associated to $\tau$. Since $\rho$ is not one-dimensional, $\phi_\tau$ is not one-dimensional, and hence $\phi_\tau^{I_F}=0$. Then we have
\begin{equation*}
\varepsilon(\pi_{0,E}\otimes\chi)=\varepsilon(\pi_0\otimes\chi|_{F^\times})\eta(-1)^n =\varepsilon(\phi_{\tau\otimes\chi|_{F^\times}})^{\ell^\prime}\eta(-1)^n = 1,
\end{equation*}
which proves the theorem in this case. 

Now we assume $\ell$ is odd. Again, by the proof of \cite[Proposition 5.6]{BroussousMatringe2021}, we have $\rho$ is $(H_s, \chi^{-1})$-distinguished, where $H_s$ is the centralizer of $E^\times$ in $\GL_s(D)$. Let $\rho^\prime=\JL(\rho)$ be the Jacquet-Langlands transfer of $\rho$ to $\GL_{2s}(F)$. Write $\rho^\prime$ as a segment
\begin{equation*}
\{\tau\nu^{-(a-1)/2}, \cdots, \tau\nu^{(a-1)/2}\},
\end{equation*}
and write $\pi_0=\JL(\pi)$ as a segment
\begin{equation*}
\{\tau\nu^{-(\ell a-1)/2}, \cdots, \tau\nu^{(\ell a-1)/2}\}.
\end{equation*}
If $a$ is even, then the same computation as in the case $\ell$ being even gives that
\begin{equation*}
     \varepsilon(\pi_{0,E}\otimes\chi)= \varepsilon(\phi_{\tau\otimes\chi|_{F^\times}})^{\ell a}\eta(-1)^n = \eta(-1)^n.
\end{equation*}
If $a$ is odd, then both $a$ and $\ell a$ are odd, and 
\begin{equation*}
     \varepsilon(\pi_{0,E}\otimes\chi)= \varepsilon(\phi_{\tau\otimes\chi|_{F^\times}})^{\ell a}\eta(-1)^n .
\end{equation*}
In both cases, we have
\begin{equation*}
     \varepsilon(\pi_{0,E}\otimes\chi)= \varepsilon(\phi_{\tau\otimes\chi|_{F^\times}})^{\ell a}\eta(-1)^n .
\end{equation*}
On the other hand, by Theorem~\ref{theorem-forward-direction-main} for the supercuspidal case, we get
\begin{equation*}
\varepsilon( \BC(\rho^\prime)\otimes\chi) = \varepsilon(\phi_{\tau\otimes\chi|_{F^\times}})^a \eta(-1)^s= \varepsilon_{D}^s \eta(-1)^s \chi(-1)^s.
\end{equation*}
Since $n=s\ell$ and $\ell$ is odd, $n$ and $s$ have the same parity, and thus we conclude that
\begin{equation*}
     \varepsilon(\pi_{0,E}\otimes\chi)= \varepsilon(\phi_{\tau\otimes\chi|_{F^\times}})^{\ell a}\eta(-1)^n = \varepsilon(\phi_{\tau\otimes\chi|_{F^\times}})^{a}\eta(-1)^n = \varepsilon_{D}^n \eta(-1)^n \chi(-1)^n .
\end{equation*}
This completes the proof of Theorem~\ref{theorem-forward-direction-main}.

\section{The converse direction}
\label{section-converse}

The goal of this section is to prove Theorem~\ref{thm-main-converse}, which we restate below. We recall the following setup:
\begin{itemize}
	\item $E/F$ is a quadratic extension of local nonarchimedean fields of characteristic zero.
	\item $D$ is a quaternion algebra over $F$ containing $E$.
	\item $G=\GL_n(D)$, and $H=\Res_{E/F}\GL_{n,F}$ regarded as a subgroup of $G$.
	\item Put $\epsilon_D=1$ (resp. $\epsilon_D=-1$) if $D$ splits (resp. ramifies).
\end{itemize}

\begin{theorem}
\label{thm-converse-for-supercuspidal}
Assume $\chi|_{F^\times}$ is trivial. Suppose $\pi_0$ is a discrete series representation of $\GL_{2n}(F)$ satisfying 
\begin{itemize}
    \item[(1)] the Langlands parameter of $\pi_0$ takes values in $\Sp_{2n}(\C)$,
    
    \item[(2)] $\varepsilon(\pi_{0, E}\otimes\chi )=\epsilon_D^n \eta(-1)^n$.
\end{itemize}
Assume in addition that $\BC(\pi_0)$ is supercuspidal. Then the Jacquet-Langlands transfer $\pi=\JL(\pi_0)$ of $\pi_0$ to $G=\GL_n(D)$ is $(H,\chi^{-1})$-distinguished. 
\end{theorem}

In the rest of this section, we assume the conditions in Theorem~\ref{thm-converse-for-supercuspidal} are satisfied.

\subsection{Global arguments}
In this subsection, we prove a globalization result.  We will usually denote a global object by a letter with an underline. Let $\ul E/\ul F$ be a quadratic extension of number fields such that it is split at all archimedean places, and there is a finite inert place $v_0$ of $\ul F$ such that $\ul E_{v_0} /\ul F_{v_0} =E/F$. Let $\ul \eta:\ul F^\times\backslash \mathbb{A}_{\ul F}^\times\to \{\pm 1\}$ be the quadratic character attached to $\ul E/\ul F$ via class field theory.
Let $\Sigma_1$ be a finite set of inert places not containing $v_0$. 
We fix another finite split place $v_1$. We globalize $\chi$ to a character $\ul \chi$ of $\A_{\ul E}^\times$ such that $\ul \chi|_{\A_{\ul F}^\times}$ is trivial. Let $\ul \psi:\ul F\backslash \A_{\ul F}\to \mathbb{C}$ be a fixed non-trivial additive character.
Put $\ul G^\prime=\Res_{\ul E/ \ul F}(\GL_{2n, \ul F})$, $\ul H^\prime=\Res_{\ul E/ \ul F}(\GL_{n, \ul F}\times\GL_{n, \ul F})$, and $\ul H^{\prime\prime}=\GL_{2n, \ul F}$.
\begin{proposition}
\label{prop-globalize-pi_0}
There is an irreducible cuspidal automorphic representation $\ul \pi_0$ of $\GL_{2n}(\A_{\ul F})$ such that
\begin{itemize}
    \item[(1)]  ${\ul \pi_0}_{v_0}=\pi_0$, ${\ul \pi_0}_{v_1}$ is supercuspidal, ${\ul \pi_0}_{v_1}\not\cong {\ul \pi_0}_{v_1}\otimes \ul\eta_{v_1}$;
    
    \item[(2)] if $v$ is a nonsplit place and ${\ul \pi_0}_{v}$ is not supercuspidal, then all data at $v$ are unramified;
    
    \item[(3)] $\BC({\ul \pi_0})$ is globally distinguished by $(\ul H^\prime, \ul \chi_{\ul H^\prime}^{-1})$ and $(\ul H^{\prime\prime}, \ul\eta)$.
\end{itemize}
\end{proposition}

The existence of such a $\ul \pi_0$ requires an argument similar to \cite[Proposition A.4]{IchinoLapidMao2017} and \cite[Theorem 16.3.2]{SakellaridisVenkatesh}, and we present it in Section~\ref{subsection-proof-globalize-pi_0}.

We also globalize CSAs. 

\begin{lemma}
There is a CSA $\ul A$ over $\ul F$ containing $\ul E$ with the following properties. \begin{itemize}
    \item[(1)] $\ul A$ splits at all split places.
    
    \item[(2)] $\ul A_{v_0}=\Mat_n(D)$.
    
    \item[(3)] If the split rank of $\ul G_v$ is $r_v$, then $\varepsilon({\ul \pi_{0,E}}_v\otimes\chi_v)=(-1)^{r_v} \ul \eta_v(-1)^n$.
 \end{itemize}  
\end{lemma}

\begin{proof}
First we assume $n$ is odd. We let $w_1, w_2, \cdots, w_{2k}$ be the places where 
$$\varepsilon({\ul \pi_{0,E}}_{w_i}\otimes\chi_{w_i}) \ul\eta_{w_i}(-1)^n=-1$$
and assume $w_1=v_0$. Then we take a CSA $A_i$ at $w_i$ such that $A_1=\Mat_n(D)$, $\mathrm{inv}(A_i)=-\mathrm{inv}(A_{n+i})$ for $i=1, \cdots, n$, and the split rank of $A_i^\times$'s are all odd (note that the split rank of $A_i^\times$ and $A_{n+i}^\times$ are the same). Then we take $\ul A$ such that $\ul A_{w_i}=A_i$ and $\ul A_v=\Mat_{2n}(\ul F_v)$ for all other places $v$. Now we assume that $n$ is even. We assume the invariant of $\Mat_n(D)$ is $a/b$, then $b|n$. We take two places $w_1, w_2$, such that 
$$\varepsilon({\ul \pi_{0,E}}_{w_i}\otimes\chi_{w_i}) \ul\eta_{w_i}(-1)^n=1$$
and take $A_1$ and $A_2$ whose invariants are $-a/2b$. Let $w_3, \cdots, w_{2k}$ be places where
$$\varepsilon({\ul \pi_{0,E}}_{w_i}\otimes\chi_{w_i}) \ul\eta_{w_i}(-1)^n=-1$$
for $3\le i \le 2k$, and we choose $A_3, \cdots, A_{2k}$ as in the odd case. Then we take $\ul A$ such that $\ul A_{w_i}=A_i$, and $\ul A_v=\Mat_{2n}(\ul F_v)$ for all other $v$.
\end{proof}

Put $\ul G=\ul A^\times$.

Let $\ul f^\prime =\otimes  f^\prime_v$ be a test function on $\ul G^\prime (\A_{\ul F})$ that is unramified if the data at that place is unramified. At the split place, we take any test function. If $v$ is a nonsplit place with some ramified data, then either $v=v_0$ or $\BC({\ul \pi_0}_v)$ is supercuspidal by Proposition~\ref{prop-globalize-pi_0}. If $v=v_0$, then by the assumption of Theorem~\ref{thm-converse-for-supercuspidal}, $\BC({\ul \pi_0})_{v_0}=\BC(\pi_0)$ is supercuspidal, hence in both cases $\BC({\ul \pi_0}_v)$ is supercuspidal. By Theorem~\ref{thm-ellipticity}, $\BC({\ul \pi_0}_v)$ is elliptic. Let  $f_{v}^{\prime\prime}$ be a test function supported in the elliptic locus such that $I_{{\ul \pi_0}_{v}}(f^{\prime\prime}_{v})\not=0$. Let 
\begin{equation*}
f_{v}^\prime=	f_{v}^{\prime\prime}+\varepsilon({\ul \pi_{0,E}}_v\otimes\chi_v) f_{v}^{\prime\prime\dagger},
\end{equation*}
where $f_{v}^{\prime\prime\dagger}$ is the involution defined in \eqref{eq-involution-Gprime}.
By Lemma~\ref{SphericalCharacters-split-lemma1}, we have
\begin{equation*}
\begin{split}
I_{\BC({\ul \pi_0}_v)}(f^{\prime}_{v}) = I_{\BC({\ul \pi_0}_v)}(f^{\prime\prime}_{v})+ \varepsilon({\ul \pi_{0,E}}_v\otimes\chi_v)^2  I_{\BC({\ul \pi_0}_v)}(f^{\prime\prime}_{v})=2 I_{\BC({\ul \pi_0}_v)}(f^{\prime\prime}_{v}) \not=0.
\end{split}
\end{equation*}
Moreover, by assumption (2) of Theorem~\ref{thm-converse-for-supercuspidal} and by the proof of \cite[Lemma 4.1]{XueZhang2022}, $O^{\ul G^{\prime}}(g,f^{\prime}_{v})=0$ if $g\in \ul G^\prime(\ul F_v)$ does not match any element in $\ul G(\ul F_v)$. 
If $v$ is any other (nonsplit) place and $\BC({\ul \pi_0}_v)$ is supercuspidal, we take $f_v^\prime$ to be an essential matrix coefficient of ${\ul \pi_0}_{v}$. By the proof of \cite[Lemma 4.1]{XueZhang2022}, again we have $O^{\ul G^{\prime}}(g,f^{\prime}_{v})=0$ if $g\in \ul G^\prime(\ul F_v)$ does not match any element in $\ul G(\ul F_v)$. Here, by the orbital integral $O^{\ul G^{\prime}}(g,f^{\prime}_{v})$ we mean $O^{\ul G^{\prime}}(g,f^{\prime}_{v, 1})$ where $f^{\prime}_{v, 1}\in \mathcal{C}_c^\infty(\ul G^\prime (\ul F_v))$ is a function such that 
\begin{equation*}
\int_{Z^\prime (\ul F_v)} f^{\prime}_{v, 1}(zg)	\omega_{\BC({\ul \pi_0})}(z)dz = f^{\prime}_{v}(g).
\end{equation*}

Let $\ul f$ be a test function on $\ul G (\A_{\ul F})$ that matches $\ul f^\prime$. By the relative trace formula identity of \cite{XueZhang2022}, we have
\begin{equation}
    I_{\BC(\ul \pi_0)}(\ul{f}^\prime) = J_{\JL(\ul \pi_0)}(\ul{f})+J_{\JL(\ul \pi_0)\otimes\ul\eta}(\ul{f}).
\end{equation}
By the choice of test function $\ul f^\prime$, we have $I_{\BC(\ul \pi_0)}(\ul{f}^\prime)\not=0$. Thus 
$$J_{\JL(\ul \pi_0)}(\ul{f})+J_{\JL(\ul \pi_0)\otimes\ul\eta}(\ul{f})\not=0.$$ 
So either $J_{\JL(\ul \pi_0)}(\ul{f})\not=0$ or $J_{\JL(\ul \pi_0)\otimes\ul\eta}(\ul{f})\not=0$. This means that either $\JL(\ul \pi_0)$ is globally $(\ul H, \ul \chi^{-1})$-distinguished, or $\JL(\ul \pi_0)\otimes\ul\eta$ is globally $(\ul H,\ul \chi^{-1})$-distinguished. We conclude that either $\pi=\JL(\ul \pi_0)_{v_0}$ is $(H,\chi^{-1})$-distinguished, or $\pi\otimes\ul\eta_{v_0}$ is $(H,\chi^{-1})$-distinguished. Since $\ul\eta_{v_0}$ is trivial on $H$, it follows that $\pi$ being $(H,\chi^{-1})$-distinguished is equivalent to $\pi\otimes\ul\eta_{v_0}$ being $(H,\chi^{-1})$-distinguished. Hence $\pi$ is $(H,\chi^{-1})$-distinguished.

\subsection{Proof of Proposition~\ref{prop-globalize-pi_0}}
\label{subsection-proof-globalize-pi_0}
We now prove Proposition~\ref{prop-globalize-pi_0}. We are given a self-dual discrete series representation $\pi_0$ of $\GL_{2n}(F)$ with Langlands parameter 
\begin{equation*}
	\rho_{\pi_0}:\WD_F\to \Sp_{2n}(\mathbb{C}),
\end{equation*}
where $\WD_F=W_F\times\SL_2(\C)$ is the Weil-Deligne group. The output we need is an irreducible cuspidal automorphic representation $\ul \pi_0$ of $\GL_{2n}(\A_{\ul F})$ satisfying the three conditions in Proposition~\ref{prop-globalize-pi_0}. We do this in several steps. Put $\Sigma=\Sigma_1\cup \{v_0\}$.

Let $X_n$ be a $(2n+1)$-dimensional vector space over $\ul F$ equipped with a non-degenerate symmetric bilinear form $(\cdot, \cdot)$ of Witt index $n$. Take a maximal isotropic subspace $X_n^+$ of $X_n$ of dimension $n$. Then we can write
\begin{equation*}
X_n=X_n^+ + \ul F e + X_n^-	
\end{equation*}
where $X_n^-$ is an isotropic subspace dual to $X_n^+$ and $e$ is an anisotropic vector orthogonal to $X_n^+ + X_n^-$. We assume that $(e, e)=1$. Let $\{e_1, \cdots, e_n\}$ and $\{e_{-1}, \cdots, e_{-n}\}$ be bases for $X_n^+$ and $X_n^-$ respectively so that
\begin{equation*}
(e_i, e_{-j})=\begin{cases}
1, &\text{ if }i=j,\\
0, &\text{ if }i\not=j.
\end{cases}
\end{equation*}
We denote the special orthogonal group $\SO(X_n)=\SO(2n+1)$ by $\widetilde{\ul G}$. When we write elements in $\widetilde{\ul G}$ as matrices, we will employ the ordered basis $\{e_{-1}, \cdots, e_{-n}, e, e_n, \cdots, e_1\}$. 

We take $\lambda\in \ul F^\times$ so that $\ul E=\ul F(\sqrt{\lambda})$. Let $L$ be a subspace of $X_n$ spanned by $e_{-n}, e, e_n$. We fix an $e_\lambda\in L(\ul F)$ such that $(e_\lambda, e_\lambda)=\lambda$.
Let $\ul P^\prime$ be the maximal parabolic subgroup of $\widetilde{\ul G}$ preserving the isotropic subspace spanned by $\{e_{-1}, e_{-2}, \cdots, e_{-n+1}\}$, with a Levi decomposition $\ul P^\prime=\ul M^\prime \ul S^\prime$, where $\ul M^\prime$ is the Levi subgroup. Then 
\begin{equation*}
\ul M^\prime=\left\{\begin{pmatrix} a & &\\ &h&\\ &&a^*\end{pmatrix}: a\in \GL_{n-1}, h\in \SO(L) \right\}	
\end{equation*}
where $a^*=J_{n-1} {}^t a^{-1} J_{n-1}$ with $J_{n-1}$ the $(n-1)\times (n-1)$ matrix with ones on the antidiagonal and zeros everywhere else, and $\ul S^\prime$ consists of elements of the form
\begin{equation*}
s^\prime=\begin{pmatrix}
1_{n-1} & A&B\\
&1_3 &A^\prime\\
& &1_{n-1}
\end{pmatrix}
\end{equation*}
where $A^\prime=-J_3 {}^t A J_{n-1}$ and ${}^t A^\prime J_3 A^\prime+J_{n-1}B+{}^t B J_{n-1}=0$. We define a character $\ul\psi_\lambda$ of $\ul S^\prime(\ul F)\backslash \ul S^\prime(\A_{\ul F})$ by
\begin{equation*}
\ul \psi_\lambda\begin{pmatrix}
1_{n-1} & A&B\\
&1_3 &A^\prime\\
& &1_{n-1}
\end{pmatrix}=\ul \psi((Ae_{\lambda}, e_{n-1})).
\end{equation*}

Let $U_{n-1}$ be the group of upper unipotent matrices in $\GL_{n-1}$. For $u\in U_{n-1}$, we denote \begin{equation*}
	\check{u}=\begin{pmatrix} u &&\\ &I_3&\\ &&u^*
\end{pmatrix} \in \ul P^\prime.
\end{equation*}
Let $\ul S=\ul S^{\prime}\ul S^{\prime\prime}$ be the unipotent subgroup of $\ul P^\prime$, where $\ul S^{\prime\prime}=\{ \check{u}:u\in \GL_{n-1}\}$. We extend the character $\psi_\lambda$ to $\ul S(\A_{\ul F})$ by 
\begin{equation*}
	\psi_{\lambda}(\check{u})=\psi(u_{1,2}+\cdots+u_{n-2,n-1}), \quad u\in U_{n-1}(\A_{\ul F}).
\end{equation*}

Let $\ul D_{\lambda}$ be the subgroup of $\widetilde{\ul G}$ defined by
\begin{equation*}
\ul D_{\lambda}=\left\{ \begin{pmatrix}
1_{n-1} &&\\&h &\\ &&1_{n-1}
\end{pmatrix}: h\in \SO(L), he_{\lambda}=e_{\lambda} \right\}.
\end{equation*}
and let $\ul R_\lambda=\ul D_{\lambda}\ul S$ be the Bessel subgroup. Note that the elements of $\ul D_{\lambda}(\A_{\ul F})$ stablizes $\psi_\lambda$ by conjugation, and that
\begin{equation*}
\ul D_\lambda(\ul F)\cong \SO(\ul E)\cong \ul E^\times /\ul F^\times.	
\end{equation*}
We define a character $\ul \nu_{\lambda,\ul \chi}$ on $\ul R_\lambda(\A_{\ul F})$ by
\begin{equation*}
\ul \nu_{\lambda,\ul \chi}(ts)=\ul \chi(t) \ul\psi_{\lambda}(s), \quad t\in \ul D_{\lambda}(\A_{\ul F}), s\in \ul S(\A_{\ul F}).
\end{equation*}

We say that a representation $\Pi$ of $\SO(2n+1, \mathbb{A}_{\ul F})$ has a global Bessel model of type $(\lambda, e_{\lambda},\ul \psi^{-1},\ul\chi^{-1})$ if the global Bessel period
\begin{equation*}
{P_{\lambda, e_{\lambda}, \ul\psi, \ul\chi} }(\phi):=\int_{{\ul R_{\lambda}}(\ul F) \backslash {\ul R_{\lambda}}(\A_{\ul F}) } \phi(r)\ul \nu_{\lambda,\ul \chi}(r)dr
\end{equation*}
is not identically zero on $\Pi$.

Since the local Gross-Prasad conjecture \cite{GrossPrasad1994} is known for $\SO(2n+1)\times\SO(2)$ (Waldspurger \cite{Waldspurger2012-LocalGrossPrasad}, Moeglin--Waldspurger \cite{MoeglinWaldspurger2012-LocalGrossPrasad}), $\pi_0$ descents to a unique representation $\sigma$ of $\SO(2n+1, F)$ which has a Bessel-model, i.e., $\Hom_{\ul R_{\lambda}(F)} (\sigma\otimes(\chi\psi), \mathbb{C})\not=0$. At a place $v\in \Sigma_1$, we also fix a supercuspidal representation $\tau_{v}$ of $\SO(2n+1, \ul F_{v})$ whose functorial lift to $\GL_{2n}(\ul F_{v})$ is supercuspidal. We may moreover assume that at the place $v_1$, the image of the funtorial lift is not isomorphic to its twist by $\ul\eta_{v_1}$.

Let $\mathcal{A}^B$ be the set of cuspidal automorphic representations $\Pi$ of $\SO(2n+1, \mathbb{A}_{\ul F})$ satisfying the following conditions:
\begin{itemize}
    \item[(1)] the functorial lifting of $\Pi$ to $\GL_{2n}(\mathbb{A}_{\ul F})$ is cuspidal,
    
    \item[(2)] $\Pi$ has a global Bessel model of type $(\lambda, e_{\lambda},\ul \psi^{-1},\ul\chi^{-1})$,
    
    \item[(3)] $\Pi$ is unramified at all inert places not in $\Sigma$.
\end{itemize}

Recall the notion of weak containment in \cite[Appendix F]{BekkaHarpeValette2008}. We will prove the following lemma.
\begin{lemma}
\label{lemma-weak-containment}
The space $L^2({\ul R_{\lambda}}_{\Sigma}\backslash \widetilde{\ul G}_{\Sigma},{\ul \nu_{\lambda,\ul \chi}^{-1}})$ is weakly contained in $\bigoplus_{\Pi\in \mathcal{A}^B}\Pi_\Sigma$. 	
\end{lemma}

Assuming Lemma~\ref{lemma-weak-containment} for now, we conclude that there is a sequence of $\Pi^{(\ell)}\in \mathcal{A}^B$, $\ell=1, 2, \cdots$, such that 
\begin{equation*}
\Pi^{(\ell)}_\Sigma \to \rho\otimes\bigotimes_{v\in \Sigma_1}\tau_v
\end{equation*}
in the Fell topology. Moreover, the same argument as in the proof of \cite[Proposition 3.6.1]{Beuzart-Plessis2021J.Inst.Math.Jussieu} gives that $\Pi^{(\ell)}_\Sigma$ is tempered when $\ell$ is large. The restriction of the Fell topology to the tempered spectrum coincides with the usual topology on the tempered spectrum (cf. Remark after \cite[Proposition 4.11]{Beuzart-Plessis2021Invent}) and square integrable representations are isolated points. It follows that
\begin{equation*}
\Pi^{(\ell)}_\Sigma = \rho\otimes\bigotimes_{v\in \Sigma_1}\tau_v 
\end{equation*}
when $\ell$ is sufficiently large. For this $\Pi^{(\ell)}$, let $\ul \pi_0$ be its functorial lifting to $\GL_{2n}(\A_{\ul F})$. 
Now we explain that $\ul \pi_0$ is the desired representation. We just need to explain condition (3) in Proposition~\ref{prop-globalize-pi_0}. Note that $\ul \pi_0$ is cuspidal and $\ul \pi_0 \not\cong \ul \pi_0\otimes \ul\eta$. Hence its base change $\BC(\ul \pi_0)$ to $\GL_{2n}(\A_{\ul E})$ exists and is unique, which is a cuspidal representation of $\GL_{2n}(\A_{\ul E})$ \cite{ArthurClozel1989}. Since $\Pi^{(\ell)}$ has a global Bessel model, by 
\cite[Theorem 5.7]{JiangZhang2020} on the global Gross-Prasad conjecture, we have $L(\frac{1}{2}, \BC(\ul \pi_0)\otimes \ul\chi)\not=0$. This implies that $\BC(\ul \pi_0)$ is globally distinguished by $(\ul H^\prime, \ul \chi_{\ul H^\prime}^{-1})$. On the other hand, since $\ul \pi_0$ comes from the functorial lifting of a cuspidal representation of $\SO({2n+1},\A_F)$, $L^S(s,\ul \pi_0, \wedge^2)$ has a simple pole at $s=1$ (cf. \cite[Theorem 1]{GinzburgRallisSoudry2001}). By \cite[Theorem 1.1]{Shahidi1997}, $L^S(s,\ul\pi_0, \Sym^2)$ is non-zero at $s=1$.
Observe that
\begin{equation*}
\begin{split}
L^S(s,\BC(\ul \pi_0)\otimes\ul \chi \ul\eta, \Asai^-) &=	 L^S(s,\ul\pi_0, \Sym^2\otimes\ul\chi|_{\A_F^\times})  L^S(s,\ul\pi_0, \wedge^2\otimes\ul\chi|_{\A_F^\times}\ul\eta) \\
&=	 L^S(s,\ul\pi_0, \Sym^2)  L^S(s,\ul\pi_0, \wedge^2\otimes\eta).
\end{split}
\end{equation*}
It follows that $L^S(s,\BC(\ul \pi_0)\otimes\ul \chi\ul\eta, \Asai^-)$ has a pole at $s=1$. Thus $\BC(\ul \pi_0)$ is globally distinguished by $(\ul H^{\prime\prime}, \ul\eta)$.

It remains to prove Lemma~\ref{lemma-weak-containment}.
\begin{proof}[Proof of Lemma~\ref{lemma-weak-containment}]
	By \cite[Lemma F.1.3]{BekkaHarpeValette2008}, it suffices to prove that for any $\varepsilon>0$, any compact $\Omega\subset \widetilde{\ul G}_{\Sigma}$, and all $\alpha\in \mathcal{C}_c^\infty ( \widetilde{\ul G}_{\Sigma})$ such that
	\begin{equation*}
		\int_{  \widetilde{\ul G}(\ul F_{\Sigma})}\int_{{\ul R_{\lambda}} (\ul F_{\Sigma})} \alpha(hg)\overline{\alpha(g)} {\ul \nu_{\lambda,\ul \chi}}(h) dhdg\not=0,
	\end{equation*}
	there are finitely many $\varphi\in \oplus_{\Pi\in \mathcal{A}^B}\Pi_\Sigma$ such that the function
	\begin{equation}
	\label{eq-SOodd-weaklycontainment}
	\left| \int_{  \widetilde{\ul G}(\ul F_{\Sigma})} \int_{{\ul R_{\lambda}}(\ul F_{\Sigma})} \alpha(hgy)\overline{\alpha(g)} {\ul \nu_{\lambda,\ul \chi}}(h)dhdg-\sum_{\varphi} \langle \Pi_{\Sigma}(y)\varphi, \varphi\rangle \right|<\varepsilon	
	\end{equation}
	for all $y\in \Omega$. 
	
	We fix a split place $w$ and a supercuspidal represnetation $\tau_w$ of $\SO(2n+1,\ul F_w)$ whose functorial transfer to $\GL_{2n}(\ul F_w)$ is still supercuspidal. Let $\phi=\otimes \phi_v\in \mathcal{C}_c^\infty(\SO(2n+1, \A_{\ul F})$ be a test function such that $\phi_\Sigma=\alpha$, $\phi_v$ is the unit in the spherical Hecke algebra if $v$ is inert and not in $\Sigma$, and a suitable test function if $v$ splits. In particular, we require $\phi_w$ is the matrix coefficient of $\tau_w$. Let
	\begin{equation*}
	K_\phi(h,g)=\sum_{\gamma\in \SO(2n+1, \ul F)} \phi(h^{-1} \gamma g)
	\end{equation*}
	be the usual kernel function and put
	\begin{equation*}
	K_\phi(g)=\int_{{\ul R_{\lambda}}(\ul F)\backslash {\ul R_{\lambda}}(\A_{\ul F})} K_\phi(h,g) \ul \nu_{\lambda,\ul \chi}(h) dh.
	\end{equation*}
We then compute
\begin{equation}
\label{eq-SOodd-traceformula}
\langle R(y)K_\phi, K_\phi\rangle.	
\end{equation}

Spectrally, we have
\begin{equation*}
K_\phi(h,g)=\sum_{\Pi}\sum_{\varphi\in \Pi}\Pi(\phi)\varphi(h)\overline{\varphi(g)},	
\end{equation*}
and
\begin{equation*}
K_\phi(g)=\sum_{\Pi\in \mathcal{A}^B}\sum_{\varphi\in \Pi} {P_{\lambda, e_{\lambda}, \ul\psi, \ul\chi} }(\Pi(\phi)\varphi))\overline{\varphi(g)}.	
\end{equation*}
Thus, 
\begin{equation}
\begin{split}
 &\langle R(y)K_\phi, K_\phi\rangle	\\
 =& \langle \sum_{\Pi\in \mathcal{A}^B}\sum_{\varphi\in \Pi} {P_{\lambda, e_{\lambda}, \ul\psi, \ul\chi} }(\Pi(\phi)\varphi))  \overline{\Pi(y) \varphi}, \sum_{\Pi\in \mathcal{A}^B}\sum_{\varphi\in \Pi} {P_{\lambda, e_{\lambda}, \ul\psi, \ul\chi} }(\Pi(\phi)\varphi))\overline{\varphi}   \rangle \\
 =& \sum_{\Pi\in \mathcal{A}^B} \sum_{\varphi\in \Pi} {P_{\lambda, e_{\lambda}, \ul\psi, \ul\chi} }(\Pi(\phi)\varphi)) \overline{{P_{\lambda, e_{\lambda}, \ul\psi, \ul\chi} }(\Pi(\phi)\varphi))}  \langle  \overline{\Pi(y) \varphi},  \overline{\varphi}\rangle.
\end{split}
\label{eq-SOodd-traceformula-spec}
\end{equation}
Here, the sum is over $\mathcal{A}^B$ because of our choices of the test functions, in particular at $w$ and the inert places. 

Geometrically, we have
\begin{equation*}
\int_{\widetilde{\ul G}(\ul F)\backslash \widetilde{\ul G}(\A_{\ul F})} \int_{[{\ul R_{\lambda}}]^2} \sum_{\gamma_1, \gamma_2\in \widetilde{\ul G}(\ul F)}{\phi(h_1^{-1}\gamma_1 gy)}\overline{\phi(h_2^{-1}\gamma_2 g)}  {\ul \nu_{\lambda,\ul \chi}}(h_1)\overline{{\ul \nu_{\lambda,\ul \chi}}(h_2)} dh_1dh_2dg.
\end{equation*}
By collapsing the sum over $\gamma_2$ with the $dg$ integration, and making a change of variable $g\mapsto h_2 g$, we obtain
\begin{equation*}
	\int_{  \widetilde{\ul G}(\A_{\ul F})} \int_{[{\ul R_{\lambda}}]^2} \sum_{\gamma \in \widetilde{\ul G}(\ul F)} 
	 {\phi(h_1^{-1}\gamma h_2 gy)} \overline{\phi(  g)} {\ul \nu_{\lambda,\ul \chi}}(h_1)\overline{{\ul \nu_{\lambda,\ul \chi}}(h_2)} dh_1dh_2dg.
\end{equation*}
We consider the ${\ul R_{\lambda}} \times {\ul R_{\lambda}} $ action on $\widetilde{\ul G} $ given by $(r_1, r_2)\cdot g=r_1^{-1} g r_2$. For $\gamma\in \widetilde{\ul G}({\ul F})$, we denote $({\ul R_{\lambda}} \times {\ul R_{\lambda}})(\ul F) \cdot \gamma$ for the orbit of $\gamma$. Then the integral is equal to
\begin{equation}
\sum_{\gamma} \int_{[(\ul R_{\lambda}\times \ul R_{\lambda})_{\gamma}]} {\ul \nu_{\lambda,\ul \chi}}(r_1r_2^{-1})dr_1 dr_2 \cdot J(\gamma, \phi). 
\label{eq-SOodd-traceformula-geo-1}
\end{equation}
Here, $\gamma$ runs over the representatives for ${\ul R_{\lambda}}(\ul F)\backslash   \widetilde{\ul G}( {\ul F}) /{\ul R_{\lambda}}(\ul F)$,  $(\ul R_{\lambda}\times \ul R_{\lambda})_{\gamma}$ is the stabilizer of $\gamma$, and
\begin{equation*}
J(\gamma, \phi) :=	\int_{  \widetilde{\ul G}(\A_{\ul F})}	 \int_{(\ul R_{\lambda}\times \ul R_{\lambda})_{\gamma}(\A_{\ul F})\backslash (\ul R_{\lambda}\times \ul R_{\lambda})(\A_{\ul F})} {\phi(h_1^{-1}\gamma h_2 gy)} \overline{\phi(  g)} {\ul \nu_{\lambda,\ul \chi}}(h_1 h_2^{-1})dh_1dh_2dg.
\end{equation*}
Note that the integral
\begin{equation*}
\int_{[(\ul R_{\lambda}\times \ul R_{\lambda})_{\gamma}]} {\ul \nu_{\lambda,\ul \chi}}(r_1r_2^{-1})dr_1 dr_2 	
\end{equation*}
vanishes unless ${\ul \nu_{\lambda,\ul \chi}}(r_1 r_2^{-1})$ is trivial on $[(\ul R_{\lambda}\times \ul R_{\lambda})_{\gamma}]$, in which case the integral is equal to  $\vol([(\ul R_{\lambda}\times \ul R_{\lambda})_{\gamma}])$. Also, we have
    \[
    J(1_{2r+1}, \phi) =	\int_{  \widetilde{\ul G}(\A_{\ul F})}	 \int_{  \ul R_{\lambda}  (\A_{\ul F})} {\phi(r gy)} \overline{\phi(  g)} {\ul \nu_{\lambda,\ul \chi}}(r)drdg,
    \]
and the right hand side further decomposes into
\begin{equation*}
\begin{split}
 \left( \int_{  \widetilde{\ul G}(\ul F_{\Sigma})} \int_{{\ul R_{\lambda}}(\ul F_{\Sigma})} \alpha(hgy)\overline{\alpha(g)} {\ul \nu_{\lambda,\ul \chi}}(h)dhdg \right) 
 \cdot \prod_{v\not\in \Sigma} \int_{  \widetilde{\ul G}(\ul F_{v})} \int_{{\ul R_{\lambda}}(\ul F_{v})} \phi_v(hgy)\overline{\phi_v(g)} {\ul \nu_{\lambda,\ul \chi}}(h)dhdg.
\end{split}
\end{equation*}
We keep all other places intact but only one $\phi_v$, and shrink its support to a sufficiently small neighborhood of identity. We can choose the support of $\phi_v$ so small that it has empty intersection with the orbit $({\ul R_{\lambda}} \times {\ul R_{\lambda}})(\ul F) \cdot \gamma$ if the orbit $({\ul R_{\lambda}} \times {\ul R_{\lambda}})(\ul F) \cdot \gamma$ does not contain the identity. Then for all $y\in \Omega$, all the terms in \eqref{eq-SOodd-traceformula-geo-1} vanish except the term corresponding to $\gamma=1_{2r+1}$. Thus, the geometric expansion gives
 \begin{equation}
 \label{eq-SOodd-traceformula-geo-2}
 \begin{aligned}
 	 &\left( \int_{  \widetilde{\ul G}(\ul F_{\Sigma})} \int_{{\ul R_{\lambda}}(\ul F_{\Sigma})} \alpha(hgy)\overline{\alpha(g)} {\ul \nu_{\lambda,\ul \chi}}(h)dhdg \right) \\
   &
   \vol([(\ul R_{\lambda}\times \ul R_{\lambda})_{1_{2r+1}}]) \cdot \prod_{v\not\in \Sigma} \int_{  \widetilde{\ul G}(\ul F_{v})} \int_{{\ul R_{\lambda}}(\ul F_{v})} \phi_v(hgy)\overline{\phi_v(g)} {\ul \nu_{\lambda,\ul \chi}}(h)dhdg.
   \end{aligned}
 \end{equation}
By comparing \eqref{eq-SOodd-traceformula-geo-2} with \eqref{eq-SOodd-traceformula-spec}, and truncating the spectral side \eqref{eq-SOodd-traceformula-spec} to a finite sum, we obtain \eqref{eq-SOodd-weaklycontainment}, as desired. This finishes the proof.
\end{proof}

\appendix
\section{Ellipticity of supercuspidal representations of $G^\prime$}
\label{section-ellipticity-of-supercuspidal}

The goal of this section is to prove Proposition~\ref{prop-representability} and Theorem~\ref{thm-ellipticity}. The argument is standard and it is very close to the classical theory of Harish-Chandra. The main part of the proof also appeared in \cite{Xue2022} in a different setting. Here, we will only sketch the main argument of the proof since a detailed proof is quite long and deviates significantly from the main goal of this paper. We refer the reader to \cite{Xue2022}, \cite{Guo1998}, \cite{Kottwitz2005} for more details of similar setups.

\subsection{The Lie algebra of $G^\prime/ H^{\prime\prime}$}
Recall that we have a symmetric space 
\begin{equation*}
    S^\prime=\left\{ g \overline{g}^{-1} | g\in G^\prime \right\} \cong G^\prime/ H^{\prime\prime}.
\end{equation*}
on which $H^\prime$ acts by twisted conjugation. Let
\begin{equation*}
\mathfrak{s}^\prime=\left\{ \begin{pmatrix} & X\\ Y &\end{pmatrix} | X, Y\in \Mat_n(E)^- \right\},
\end{equation*}
where $\Mat_n(E)^-$ denotes the matrices in $\Mat_n(E)$ with purely imaginary entries. This is viewed as an algebraic variety over $F$, and it is isomorphic to the tangent space of $S^\prime$ at the point represented by the identity element in $G^\prime$. The stabilizer of $1$ in $H^\prime$ is isomorphic to $\GL_n(F)\times\GL_n(F)$, which acts on $\mathfrak{s}^\prime$ by conjugation. 

Let $g^\prime\in G^\prime$ and let $s^\prime=g\overline{g}^{-1}\in S^\prime$. The tangent space of $S^\prime$ at the point $s^\prime$ is identified with
\begin{equation*}
T_{s^\prime}=\{ gY\overline{g}^{-1}|Y\in \mathfrak{s}^\prime\}.
\end{equation*}
The tangent space of the $H^\prime$-orbit of $s^\prime$ is identified with 
\begin{equation*}
TO_{s^\prime}=\{Xs^\prime-s^\prime \overline{X}|X\in \mathfrak{h}^\prime\}.
\end{equation*}
We fix an inner product on $T_{s^\prime}$ via
\begin{equation*}
\langle gY_1\overline{g}^{-1}, gY_2\overline{g}^{-1}\rangle =\tr Y_1 Y_2.
\end{equation*}
Note that this inner product is $H_{s^\prime}^\prime$-invariant. Let
\begin{equation*}
N_{s^\prime}=\{ Ys^\prime|\theta^\prime(Y)=-Y, Ys^\prime=-s^\prime \overline{Y}\}.
\end{equation*}
Then we have an orthogonal decomposition
\begin{equation*}
T_{s^\prime}=TO_{s^\prime}\oplus N_{s^\prime}
\end{equation*}
and hence $N_{s^\prime}$ is the sliced representation at $s^\prime$.

Now we describe more concretely the sliced representations at $s^\prime$. Let $s^\prime=s^\prime(\alpha, n_1, n_2, n_3)$ be a semisimple element in $S^\prime$. The stabilizer of $s^\prime(\alpha, n_1, n_2, n_3)$ in $H^\prime$ is equal to 
\begin{equation*}
    H_1^\prime\times H_2^\prime\times H_3^\prime=(\GL_{n_1, E})_{\alpha, \mathrm{twisted}} \times \GL_{n_2, E}\times (\GL_{n_3, F}\times \GL_{n_3, F}).
\end{equation*}
The sliced representation at $s^\prime(\alpha, n_1, n_2, n_3)$ can be identified as $V_1^\prime\oplus V_2^\prime\oplus V_3^\prime$ where
\begin{equation*}
V_1^\prime=\{A\in \Mat_{n_1}(E): \alpha \overline{A}=A\overline{\alpha}\}, \quad V_2^\prime=\Mat_{n_2}(E), \quad V_3^\prime=\Mat_{n_3}(E)^{-}\oplus \Mat_{n_3}(E)^{-}.
\end{equation*}
We have the following descriptions in the three extreme cases where $n_1=n$ (Case (i)), $n_2=n$ (Case (ii)), and $n_3=n$ (Case (iii)).

\begin{itemize}
    \item[(i)] Assume $n_1=n, n_2=n_3=0$. Then $s^\prime=\begin{pmatrix}
        \alpha & 1\\
        1-\alpha \overline{\alpha} & -\overline{\alpha}
    \end{pmatrix}$. The embedding  of $(\GL_{n_1, E})_{\alpha, \mathrm{twisted}}$ in $H^\prime$ is given by
    \begin{equation*}
        h\mapsto \begin{pmatrix} h & \\ & \overline{h}\end{pmatrix}.
    \end{equation*}
    The embedding of $V_1^\prime$ in $N_{s^\prime}$ is given by
    \begin{equation*}
        A\mapsto \begin{pmatrix}
            & A \\
            -\overline{A}(1-\alpha\overline{\alpha}) & 
        \end{pmatrix}s^\prime.
    \end{equation*}
    
    \item[(ii)] Assume $n_2=n, n_1=n_3=0$. Then $s^\prime=\begin{pmatrix}
        & 1\\
        1&
    \end{pmatrix}$. The embedding of $\GL_{n_2, E}=\GL_{n, E}$ in $H^\prime$ is given by
    \begin{equation*}
        h\mapsto \begin{pmatrix} h & \\ & \overline{h}\end{pmatrix}.
    \end{equation*}
    The embedding of $V_2^\prime=\Mat_n(E)$ into $N_{s^\prime}$ is given by
    \begin{equation*}
         A\mapsto \begin{pmatrix}
            & A \\
            -\overline{A} & 
        \end{pmatrix}s^\prime=\begin{pmatrix}
             A & \\
            &-\overline{A}
        \end{pmatrix}.
    \end{equation*}

     \item[(iii)] Assume $n_3=n, n_1=n_2=0$. Then $s^\prime=1$. The embedding of $\GL_{n_3, F}\times \GL_{n_3, F}=\GL_{n, F}\times \GL_{n, F}$ in $H^\prime$ is given by 
     \begin{equation*}
         (h_1, h_2)\mapsto \begin{pmatrix}
             h_1& \\ & h_2
         \end{pmatrix}.
     \end{equation*}
     The embedding of $V_3^\prime=\Mat_n(E)^-\oplus \Mat_n(E)^-$ into $N_{s^\prime}$ is given by
     \begin{equation*}
         (A, B)\mapsto \begin{pmatrix}
             & A\\ B &
         \end{pmatrix} s^\prime=\begin{pmatrix}
             & A\\ B &
         \end{pmatrix}.
     \end{equation*}
\end{itemize}

\subsection{Semisimple descent of orbital integrals}

We recall the notion of the analytic Luna slice. This is a very general notion, so temporarily we let $G$ be a reductive group which acts on a smooth affine variety $X$. Let $x\in X(F)$ be $G$-semisimple (i.e., the $G$-orbit of $x$ is Zariski closed, or equivalently $G(F)x\subset X(F)$ is closed in the analytic topology, see \cite[Theorem 2.3.8]{AizenbudGourevitch2009}). Let $N_{Gx, x}^X$  be the normal space of $Gx$ at $x$. Write 
$N_x=N_{Gx, x}^X$.
Then by \cite[Theorem 2.3.17]{AizenbudGourevitch2009}, there exists the analytic Luna slice at $x$, denoted by $(U, p, \psi, M, N_x)$ (or simply by $(U, p, \psi)$), where 
\begin{itemize}
    \item $U$ is an open $G(F)$-invariant analytic neighborhood of $x$ in $X(F)$.
    \item $p$ is an $G(F)$-equivariant analytic retraction $p:U\to G(F)x$, and $M=p^{-1}(x)$.
    \item $\psi$ is an $G(F)_x$-equivariant analytic embedding $M\to N_x$ with an open saturated image  such that $\psi(x)=0$
\end{itemize}
Here saturated means that $M=\psi^{-1}(\psi(M))$. Let $y\in p^{-1}(x)$ and $z:=\psi(y)$. Then we have (see \cite[Corollary 2.3.19]{AizenbudGourevitch2009})
\begin{itemize}
    \item $(G_x)_z=G_y$.\\
    \item $N_{Gy, y}^X=N_{G_x(F)z, z}^{N_x}$ as $G(F)_y$ spaces.\\
    \item $y$ is $G$-semisimple if and only if $z$ is $G_x$-semisimple.
\end{itemize}

Now we describe the semisimple descent of orbital integrals. 

\begin{proposition}[\cite{ZhangWei2014}]
Let $\chi$ be a character of $G(F)$. Let $x\in X(F)$ be $G$-semisimple, and let $(U, p, \psi, M, N_x)$ be an analytic Luna slice at $x$. Then there exists a neighborhood $\mathcal{U}\subset \psi(M)$ of $0$ in $N_x$ with the following properties:
\begin{itemize}
    \item To every $f\in \mathcal{C}_c^\infty(X(F))$, there is an $f_x\in \mathcal{C}_c^\infty(N_x(F))$ such that for all semisimple $z\in \mathcal{U}$ with $z=\psi(y)$ such that $\chi$ is trivial on $G_y(F)$, we have
    \begin{equation}
        \int_{G_y(F)\backslash G(F)}f(gy)\chi(g)dg=\int_{G_y(F)\backslash G_x(F)}f_x(gz)\chi(g)dg.
    \label{eq-semisimple-descent}
    \end{equation}

    \item Conversely, given any $f_x\in \mathcal{C}_c^\infty(N_x(F))$, there exists an $f\in \mathcal{C}_c^\infty(X(F))$ such that \eqref{eq-semisimple-descent} holds if $z\in \mathcal{U}$ is semisimle, $z=\psi(y)$, and $\chi$ is trivial on $G_y(F)$.
\end{itemize}
\label{prop-semisimple-descent}
\end{proposition}

\subsection{Orbital integrals on the sliced representations}

In this subsection, we define orbital integrals on the sliced representations. 
Let $s^\prime=s^\prime(\alpha, n_1, n_2, n_3)$ be a semisimple element in $S^\prime$. Recall that the sliced representation at $s^\prime$ is isomorphic to 
\begin{equation*}
    (H_1^\prime, V_1^\prime) \times (H_2^\prime, V_2^\prime) \times (H_3^\prime, V_3^\prime)
\end{equation*}
where
\begin{equation*}
    H_1^\prime =(\GL_{n_1, E})_{\alpha, \mathrm{twisted}}, \quad H_2^\prime = \GL_{n_2, E},\quad H_3^\prime= \GL_{n_3, F}\times \GL_{n_3, F}.
\end{equation*}
and
\begin{equation*}
V_1^\prime=\{A\in \Mat_{n_1}(E): \alpha \overline{A}=A\overline{\alpha}\}, \quad V_2^\prime=\Mat_{n_2}(E), \quad V_3^\prime=\Mat_{n_3}(E)^{-}\oplus \Mat_{n_3}(E)^{-}.
\end{equation*}
We can speak of orbital integrals on each component.

\subsubsection{Orbital integrals on $(H_1^\prime, V_1^\prime)$}
The group $H_1^\prime$ acts on $V_1^\prime$ by twisted conjugation. Let $L$ be the centralizer of $\alpha\overline{\alpha}$ in $\GL_{n_1}(F)$. Then $H_1^\prime$ is an inner form of $L$. The map
$V_1^\prime\to \mathfrak{h}_1^\prime$ given by $X\mapsto X\overline{\alpha}$ is an isomorphism of representations of $H_1^\prime$, where $H_1^\prime$ acts on $\mathfrak{h}_1^\prime$ by conjugation. We may speak of semisimple and regular semisimple elements in $V_1^\prime$. For regular semisimple $X\in V_1^\prime$ and any test function $f^\prime\in \mathcal{C}_c^\infty(V_1)$, we define an orbital integral
\begin{equation*}
    O^{V_1^\prime}(X, f^\prime)=\int_{H_{1, X}^\prime\backslash H_1^\prime} f^\prime( h^{-1}X \overline{h})\mathrm{d} h
\end{equation*}
where $H_{1, X}^\prime=\{h\in H_1^\prime|h^{-1}X \overline{h}=X\}$.

\subsubsection{Orbital integrals on $(H_2^\prime, V_2^\prime)$}
Note that $V_2^\prime$ is isomorphic to $\Mat_{n_2}(E)$, and $H_2^\prime$ acts on $V_2^\prime$ by twisted conjugation. An element $X\in V_2^\prime$ is regular semisimple if $X\overline{X}$ is in $\GL_{n_2}(E)$ and is regular semisimple in the usual sense. For regular semisimple $X\in V_2^\prime$ and any test function $f^\prime\in \mathcal{C}_c^\infty(V_2^\prime)$, we define an orbital integral 
\begin{equation*}
    O^{V_2^\prime}(X, f^\prime)=\int_{H_{2, X}^\prime\backslash H_2^\prime} f^\prime(h^{-1}X \overline{h})\chi(h^{-1}\overline{h})\mathrm{d}h
\end{equation*}
where $H_{2, X}^\prime=\{h\in H_2^\prime|h^{-1}X \overline{h}=X\}$.

\subsubsection{Orbital integrals on $(H_3^\prime, V_3^\prime)$}
The group $H_3^\prime$ acts on $V_3^\prime$ by conjugation. 
An element $(X_1, X_2)\in V_3^\prime$ is semisimple or regular semisimple if $\begin{pmatrix} & X_1\\ X_2 \end{pmatrix}$ is so in $\Mat_{2n_3}(E)$. 
For $(X_1, X_2)\in V^\prime_{3}$ regular semisimple and any $f^\prime\in \mathcal{C}_c^\infty(V_3^\prime)$, we define an orbital integral
\begin{equation*}
    O^{V_3^\prime}((X_1, X_2), f^\prime)=\int_{H_{3, (X_1, X_2)}^\prime\backslash H_3^\prime} f^\prime(h_1^{-1}X_1 h_2, h_2^{-1}X_2 h_1)\eta(h_1 h_2)\mathrm{d} h_1 dh_2.
\end{equation*}
Here, $H_{3, (X_1, X_2)}^\prime=\{(h_1, h_2)\in H_3^\prime| (h_1^{-1}X_1 h_2, h_2^{-1}X_2 h_1)=(X_1, X_2)\}$.

\subsection{Outline of the proof of Proposition~\ref{prop-representability}}

Recall that for an elliptic element $g\in G^\prime$ we have defined an orbital integral
\begin{equation*}
    O^{G^{\prime}}(g, f^{\prime}) :=\int_{\left(H^{\prime} \times H^{\prime \prime}\right)_g\left(F\right) \backslash\left(H^{\prime} \times H^{\prime \prime}\right)\left(F\right)} f^{\prime} (h^{-1} g h^{\prime \prime} )\left(\chi_{H^{\prime}} \chi^{-1} \widetilde{\eta}^{-1}\right)(h)(\chi \widetilde{\eta})^{-1} (h^{-1} g h^{\prime \prime} ) \mathrm{d} h \mathrm{~d} h^{\prime \prime}.
\end{equation*}

Let $\mathcal{D}(G^\prime)^{(H^\prime, \chi_{H^{\prime}}), (H^{\prime\prime}, \chi \widetilde{\eta})}$ be the space of left $(H^\prime, \chi_{H^{\prime}} )$
 and right $(H^{\prime\prime}, \chi \widetilde{\eta})$-invariant distributions on 
$G^\prime$. It's clear that $O^{G^{\prime}}(g, \cdot)\in \mathcal{D}(G^\prime)^{(H^\prime, \chi_{H^{\prime}}), (H^{\prime\prime}, \chi \widetilde{\eta})}$.

\begin{proposition}
The set $\{O^{G^{\prime}}(g, \cdot): g\in G^\prime\text{ is regular semisimple} \}$ is weakly dense in $\mathcal{D}(G^\prime)^{(H^\prime, \chi_{H^{\prime}}), (H^{\prime\prime}, \chi\widetilde{\eta})}$; i.e., if $f^\prime\in\mathcal{C}_c^\infty(G^\prime)$ and $O^{G^{\prime}}(g, f^{\prime})=0$ for all regular semisimple $g\in G^\prime$, then $\lambda(f^\prime)=0$ for all $\lambda\in \mathcal{D}(G^\prime)^{(H^\prime, \chi_{H^{\prime}}), (H^{\prime\prime}, \chi\widetilde{\eta})}$.
\label{Prop-ellipticity-prop1}
\end{proposition}

Recall that 
\begin{equation*}
\mathfrak{s}^\prime=\left\{ \begin{pmatrix} & X\\ Y &\end{pmatrix} | X, Y\in \Mat_n(E)^- \right\},
\end{equation*}
where $\Mat_n(E)^-$ denotes the matrices in $\Mat_n(E)$ with purely imaginary entries. This is viewed as an algebraic variety over $F$, and it is isomorphic to the tangent space of $S^\prime$ at the point represented by the identity element in $G^\prime$. The stabilizer of $1$ in $H^\prime$ is isomorphic to $\mathsf{H}^\prime=\GL_n(F)\times\GL_n(F)$, which acts on $\mathfrak{s}^\prime$ by conjugation. 
If we identify $\mathfrak{s}^\prime$ with $\Mat_n(E)^-\oplus \Mat_n(E)^-$, then we have an action of $\GL_n(F)\times\GL_n(F)$ on $\mathfrak{s}^\prime$ by
\begin{equation*}
(h_1, h_2)\cdot (X, Y)=(h_1 X h_2^{-1}, h_2Yh_1^{-1}).
\end{equation*}
An element in $\mathfrak{s}^\prime$ is called semisimple or regular semisimple if it is so in $\Mat_{2n}(E)$. The locus of semisimple and regular semisimple elements in $\mathfrak{s}^\prime$ are denoted by $\mathfrak{s}^\prime_{\mathrm{ss}}$ and $\mathfrak{s}^\prime_{\mathrm{reg}}$ respectively.

Given $\gamma=(X, Y)\in \mathfrak{s}^\prime_{\text{ss}}$ and $f^\prime\in \mathcal{C}_c^\infty(\mathfrak{s}^\prime)$, we define an orbital integral
\begin{equation*}
O(\gamma, \eta, f^\prime)=\int_{\mathsf{H}^\prime_\gamma \backslash \mathsf{H}^\prime}f^\prime(h_1Xh_2^{-1}, h_2Yh_1^{-1})\eta(h_1h_2)dh_1dh_2, 
\end{equation*}
where $\mathsf{H}^\prime_\gamma=\{(h_1, h_2)\in \mathsf{H}^\prime | (h_1, h_2)\cdot \gamma =\gamma\}$. This integral is absolutely convergent. 

Let $\mathcal{D}(\mathfrak{s}^\prime)^{\mathsf{H}^\prime, \eta}$ be the space of $(\mathsf{H}^\prime, \eta)$-invariant distributions on 
$\mathfrak{s}^\prime$. Then $O(\gamma, \eta, \cdot) \in \mathcal{D}(\mathfrak{s}^\prime)^{\mathsf{H}^\prime, \eta}$ for all regular semisimple $\gamma$ in $\mathfrak{s}^\prime$.  

We fix an $\mathsf{H}^\prime$-invariant inner product on $\mathfrak{s}^\prime$ via $\langle Y_1, Y_2\rangle =\tr Y_1 Y_2$,
where the product and trace on the right-hand side are taken in $\Mat_{2n}(E)$. For $f^\prime\in \mathcal{C}_c^\infty(\mathfrak{s}^\prime)$, we define its Fourier transform by
\begin{equation*}
    \widehat{f^\prime}(X)=\int_{\mathfrak{s}^\prime} f^\prime(Y)\langle X, Y\rangle dY.
\end{equation*}
Hence we can speak of the Fourier transform of distributions on $\mathfrak{s}^\prime$.

\begin{proposition}
Let $\gamma\in \mathfrak{s}^\prime$ be regular semisimple. The Fourier transform of the distribution $O(\gamma, \eta, \cdot)$ is represented by a locally integrable  $(\mathsf{H}^\prime, \eta)$-invariant function on $\mathfrak{s}^\prime$. This function is locally constant on $\mathfrak{s}^\prime_{\mathrm{reg}}$.
\label{Prop-ellipticity-prop2}
\end{proposition}

We will define nilpotent orbital integrals of $\mathfrak{s}^\prime$ and prove the following result. 

\begin{proposition}
The Fourier transforms of nilpotent orbital integrals are represented by locally integrable functions on $\mathfrak{s}^\prime$. These functions are locally constant on $\mathfrak{s}^\prime_{\mathrm{reg}}$.
\label{Prop-ellipticity-prop3}
\end{proposition}

Proposition~\ref{prop-representability} will follow from Proposition~\ref{Prop-ellipticity-prop3}. We will give more details in Subsections \ref{subsection-nilpotent-cone}-\ref{subsection-Spherical-character-on-Gprime}. In particular, since the tangent space $\mathfrak{s}^\prime$ is isomorphic to the tangent space treated in \cite{Xue2022}, the results in Subsections \ref{subsection-nilpotent-cone}-\ref{subsection-Linear-Independence-Shalika-germs} summarize \cite{Xue2022}.

\subsection{The nilpotent cone}
\label{subsection-nilpotent-cone}
Let $\mathcal{N}\subset\mathfrak{s}^\prime$ be the nilpotent cone in $\mathfrak{s}^\prime$, which is the closed subvariety of $\mathfrak{s}^\prime$ consisting of all elements whose orbit closure contains $0\in \mathfrak{s}^\prime$. We call an element or an $\mathsf{H}^\prime$-orbit in $\mathcal{N}$ a nilpotent element or a nilpotent orbit respectively. In this subsection, we classify the nilpotent orbits which support an  $(\mathsf{H}^\prime, \eta)$-invariant distribution. 

\begin{lemma} 
The nilpotent cone $\mathcal{N}$ consists of elements in $\mathfrak{s}^\prime$ that are nilpotent in $\Mat_{2n}(E)$ in the usual sense. 
\end{lemma}

Let $V=V^+\oplus V^-$ be a $\Z/2\Z$-graded $F$-vector space with homogeneous components $V^{\pm}$ and $\dim_F V^{\pm }=n$. Then 
\begin{equation*}
    \mathfrak{s}^\prime\cong \Hom(V^+, V^-)\oplus \Hom(V^{-}, V^+), \quad \mathsf{H}^\prime\cong \GL(V^+)\times\GL(V^-).
\end{equation*}
The nilpotent cone in $\mathfrak{s}^\prime$ consists of pairs of endomorphism $\xi=(X, Y)\in \End(V)$ with $X\in \Hom(V^+, V^-)$ and $Y\in \Hom(V^{-}, V^+)$ such that $XY$ is nilpotent (and hence $YX$ is also nilpotent).

Let $\theta\in \mathsf{H}^\prime$ be the element which acts on $V^{\pm }$ by $\pm 1$. Then $\theta$ acts on $\mathfrak{gl}(V)$ by sending $Z\in \mathfrak{gl}(V)$ to $\Ad(\theta)Z=\theta Z \theta^{-1}$. Then $\mathfrak{h}$ and $\mathfrak{s}^\prime$ are eigenspaces of $\Ad(\theta)$ with eigenvalues $1$ and $-1$ respectively. 

Let $\xi=(X, Y)\in \mathcal{N}$ with $\xi^s=0$. Then we have the following filtration
\begin{equation}
    0=W_0\subset W_1\subset \cdots \subset W_{s-1}\subset W_s=V
\label{eq-ellipticy-NilpotentCone-eq1}
\end{equation}
where $W_i=\ker \xi^i$. Note that $V$ can be viewed as an $F[\xi]$-module and it is a direct sum of indecomposable $F[\xi]$-modules. By \cite{KraftProcesi1979}, we can choose the generators of these submodules to be homogeneous. Let $U$ be such an indecomposable submodule of dimension $a$ over $F$. We can choose a homogeneous element $u\in U$ such that $u, \xi u, \xi^2 u, \cdots, \xi^{a-1}u$ form a $F$-basis of $U$. Then for each $i$, we have
\begin{equation*}
    W_i=W_i^{+}\oplus W_i^{-}, \quad W_i^{\pm}=W_i\cap V^{\pm}.
\end{equation*}
It follows that we have two filtrations
\begin{equation}
    0=W_0^{\pm}\subset W_1^{\pm}\subset \cdots \subset W_{s-1}^{\pm}\subset W_s^{\pm}=V^{\pm}.
\label{eq-ellipticy-NilpotentCone-eq2}
\end{equation}
We remark that the filtration in \eqref{eq-ellipticy-NilpotentCone-eq1} is strictly increasing while the two filtrations in \eqref{eq-ellipticy-NilpotentCone-eq2} may not be strictly increasing. 

Let $r_i^?=\dim_F W_i^?/W_{i-1}^?$ where $?$ stands for $+, -$, or empty. Note that $r_i\ge r_{i+1}$ for all $i$ because $\xi$ induces an injective map $W_{i+1}/W_i\to W_i/W_{i-1}$. We also have $r_i^{\pm}\ge r_{i+1}^{\mp}$ for all $i$. Let $P=MN$ be the parabolic subgroup of $\GL(V)$ stabilizing the filtration in \eqref{eq-ellipticy-NilpotentCone-eq1}, and let $P^+=M^+N^+$ be the parabolic subgroup of $\mathsf{H}^\prime$  stabilizing both filtrations in \eqref{eq-ellipticy-NilpotentCone-eq2}. We have
\begin{equation*}
    M^+\cong \prod_{i=0}^{s-1}\GL(W_{i+1}^+/W_i^+)\times \prod_{i=0}^{s-1}\GL(W_{i+1}^-/W_i^-),
\end{equation*}
and
\begin{equation*}
    P\cap \mathsf{H}^\prime=P^+, \quad M\cap \mathsf{H}^\prime=M^+, \quad N\cap \mathsf{H}^\prime=N^+.
\end{equation*}

We have the following classification of nilpotent orbits.

\begin{lemma}
The set of nilpotent orbits in $\mathcal{N}$ is in one-to-one correspondence with the set of two sequences of integers $r_{i}^{\pm}$ with $i=1, \cdots, s$ such that 
\begin{equation}
    n=r_1^{\pm}+\cdots +r_s^{\pm}, \quad r_1^{\pm}\ge r_2^{\mp}\ge r_3^{\pm}\ge \cdots, \quad r_1^+ + r_1^- >  r_2^+ + r_2^- >\cdots > r_s^+ + r_s^- >0.
\label{eq-ellipticy-NilpotentCone-eq3}
\end{equation}
\end{lemma}

Note that we have two chains of injective maps induced by $\xi$:
\begin{equation}
    W_s^\epsilon/W_{s-1}^\epsilon \xhookrightarrow{} \cdots \xhookrightarrow{} W_3^{\mp}/W_2^{\mp} \xhookrightarrow{} W_2^{\pm}/W_1^{\pm} \xhookrightarrow{} W_1^{\mp},
\label{eq-ellipticy-NilpotentCone-eq4}
\end{equation}
where $\epsilon=+$ or $-$ depending on the parity of $s$. We call an integer $i$ with $1\le i \le s-1$ a jump if $\dim_F W_{i+1}^{\pm}/W_{i}^{\pm}<\dim_F W_{i}^{\mp}/W_{i-1}^{\mp}$ (the inequality hold for at least the $+$ one or the $-$ one, it does not have to hold for both filtrations). We call the integer $s$ a jump if $\dim_F W_s^\epsilon/W_{s-1}^\epsilon\not=0$. The following lemma establishes a necessary condition for an orbit to support an $(\mathsf{H}^\prime, \eta)$-invariant distribution.

\begin{lemma}
Suppose that the orbit represented by $\xi$ supports an $(\mathsf{H}^\prime, \eta)$-invariant distribution. Then all jumps are even integers. This means that we have the strict inequality $r_i^\epsilon>r_{i+1}^{\epsilon}$ ($\epsilon=+$ or $-$) in \eqref{eq-ellipticy-NilpotentCone-eq3} only when $i$ is even.
\label{lemma-ellipticy-NilpotentOrbit-necessary-condition}
\end{lemma}

Our next goal is to show that the condition in the above lemme is also sufficient for an orbit to support an $(\mathsf{H}^\prime, \eta)$-invariant distribution.
Let $\mathcal{O}$ be a nilpotent orbit in $\mathfrak{s}^\prime$ represented by an element $\xi$. Then attached to $\xi$, we have a parabolic subgroup $P=MN$ of $\GL(V)$, a parabolic subgroup $P^+=M^+N^+$ of $\mathsf{H}^\prime$, and two sequences of integers $r_i^{\pm}$ satisfying \eqref{eq-ellipticy-NilpotentCone-eq3}. Let $2i_1<\cdots <2i_a$ be the set of all jumps in the sequence $r_1^+\ge r_2^-\ge \cdots$, and let $2j_1<\cdots < 2j_b$ be the set of all jumps in the sequence $r_1^-\ge r_2^+\ge \cdots$. The space $\mathfrak{n}\cap \mathfrak{s}^\prime/[\mathfrak{n}, \mathfrak{n}]$ is isomorphic to 
\begin{equation*}
\begin{split}
 \bigoplus_{i=1}^{2i_a} &\Hom(W_{i+1}^{(-1)^i}/W_{i}^{(-1)^i}, W_{i}^{(-1)^{i-1}}/W_{i-1}^{(-1)^{i-1}}) \\
 &\oplus \bigoplus_{i=1}^{2j_b}\Hom(W_{i+1}^{(-1)^{i+1}}/W_{i}^{(-1)^{i+1}}, W_{i}^{(-1)^{i}}/W_{i-1}^{(-1)^{i}}).
\end{split}
\end{equation*}
We write an element in $\mathfrak{n}\cap \mathfrak{s}^\prime/[\mathfrak{n}, \mathfrak{n}]$ as 
\begin{equation*}
    m(x_1, \cdots, x_{2i_a}; y_1, \cdots, y_{2j_b})
\end{equation*}
with 
\begin{equation*}
\begin{split}
    &x_i\in \Hom(W_{i+1}^{(-1)^i}/W_{i}^{(-1)^i}, W_{i}^{(-1)^{i-1}}/W_{i-1}^{(-1)^{i-1}}),\\
    &y_i\in \Hom(W_{i+1}^{(-1)^{i+1}}/W_{i}^{(-1)^{i+1}}, W_{i}^{(-1)^{i}}/W_{i-1}^{(-1)^{i}}).
\end{split}
\end{equation*}
Note that if $i$ is odd, then both $r_{i+1}^{\pm}=r_i^{\mp}$
since all jumps are even integers. Moreover, the map induced by $\xi$
\begin{equation*}
    \xi|_{W_{i+1}^{\pm}/W_i^{\pm}}: W_{i+1}^{\pm}/W_i^{\pm}\to W_i^{\mp}/W_{i+1}^{\mp}
\end{equation*}
is an isomorphism. To ease notation, we denote $\xi_i^{\mp}= \xi|_{W_{i+1}^{\pm}/W_i^{\pm}}$. Put
\begin{equation*}
    \mathrm{det}_{2i-1}^+(x_{2i-1})=\mathrm{det} x_{2i-1}(\xi_{2i-1}^+)^{-1}, \quad \mathrm{det}_{2i-1}^-(y_{2i-1})=\mathrm{det} y_{2i-1}(\xi_{2i-1}^-)^{-1},
\end{equation*}
and
\begin{equation*}
\mathrm{det}_{\mathfrak{n}}(m)=\mathrm{det}_1^+(x_1)\mathrm{det}_3^+(x_3)\cdots \mathrm{det}^+_{2i_a-1}(x_{2i_a-1})\mathrm{det}_1^-(y_1)\mathrm{det}_3^-(y_3)\cdots \mathrm{det}^-_{2j_b-1}(y_{2j_a-1}).
\end{equation*}

Let $\mathfrak{n}^\prime$ be the subspace of $\mathfrak{n}\cap \mathfrak{s}^\prime$ generated by $[\mathfrak{n}, \mathfrak{n}]\cap \mathfrak{s}^\prime$ and
\begin{equation*}
    \bigoplus_{i \text{ even}} \Hom(W_{i+1}^+/W_i^+, W_i^-/W_{i-1}^-)\oplus \bigoplus_{i \text{ even}} \Hom(W_{i+1}^-/W_i^-, W_i^+/W_{i-1}^+).
\end{equation*}
For $f^\prime\in \mathcal{C}_c^\infty(\mathfrak{s}^\prime)$, we define a function $\widetilde{f^\prime}\in \mathcal{C}_c^\infty (\mathfrak{n}\cap \mathfrak{s}^\prime/\mathfrak{n}^\prime)$ as
\begin{equation}
\widetilde{f^\prime}(m)=\int_{\mathfrak{n}^\prime}f^\prime(m+u)du.
\label{eq-ellipticy-NilpotentCone-eq5}
\end{equation}
This is a function in the variables $m=(x_1, x_3, \cdots, x_{2i_a-1}; y_1, y_3, \cdots, y_{2j_b-1})$. Let $\underline{s}=(s_1, s_3, \cdots, s_{2i_a-1})$ and $\underline{t}=(t_1, t_3, \cdots, t_{2j_b-1})$ be complex numbers, and define
\begin{equation*}
\begin{split}
    \mathrm{det}_{\mathfrak{n}, \underline{s}, \underline{t}}(m)=&|\mathrm{det}_1^+(x_1)|^{s_1}|\mathrm{det}_3^+(x_3)|^{s_3}\cdots |\mathrm{det}_{i_a-1}^+(x_{2i_a-1})|^{s_{2i_a-1}}\\
    &|\mathrm{det}_1^-(y_1)|^{t_1}|\mathrm{det}_3^-(y_3)|^{t_3}\cdots |\mathrm{det}_{j_b-1}^-(y_{2i_a-1})|^{t_{2j_b-1}}.
\end{split}
\end{equation*}
Consider the integral
\begin{equation*}
    Z(\underline{s}, \underline{t}, \eta, \widetilde{f^\prime})=\int \widetilde{f^\prime}(m) \eta(\mathrm{det}_{\mathfrak{n}}(m))\mathrm{det}_{\mathfrak{n}, \underline{s}, \underline{t}}(m)dm,
\end{equation*}
where the domain of integration is $\mathfrak{n}\cap \mathfrak{s}^\prime/\mathfrak{n}^\prime$, which is identified with 
\begin{equation*}
    \bigoplus_{i \text{ odd}} \Hom(W_{i+1}^-/W_i^-, W_i^+/W_{i-1}^+)\oplus \bigoplus_{i \text{ odd}} \Hom(W_{i+1}^+/W_i^+, W_i^-/W_{i-1}^-) .
\end{equation*}
The integral $Z(\underline{s}, \underline{t}, \eta, \widetilde{f^\prime})$ is convergent when the real parts of $s_i$ and of $t_i$ are large enough, and it has a meromorphic continuation to $\C^{i_a+j_b}$, which is holomorphic at the points where all $s_i$'s and $t_i$'s are integers. Put
\begin{equation*}
    \widetilde{\mu}_\mathcal{O}(f^\prime)=Z(\underline{s}, \underline{t}, \eta, \widetilde{f^\prime})\biggr\rvert_{\substack{ s_i=r_i^- \text{ for all }i\\ t_i=r_i^+ \text{ for all }i}}.
\end{equation*}
Then for any $f^\prime\in \mathcal{C}_c^\infty(\mathfrak{s}^\prime)$ and any $p\in P^+$, we have
\begin{equation*}
     \widetilde{\mu}_\mathcal{O}(\Ad(p)f^\prime)=\delta_{P^+}(p)\eta(\det(p))  \widetilde{\mu}_\mathcal{O}(f^\prime).
\end{equation*}

Now we choose an open compact subgroup $K$ of $\mathsf{H}^\prime$ so that $\mathsf{H}^\prime=P^+K$. Define
\begin{equation*}
    f^\prime_K(\gamma)=\int_K f^\prime(\gamma^k)\eta(\det(k))dk, \quad \mu_\mathcal{O}(f^\prime)=\widetilde{\mu}_\mathcal{O}(f^\prime_K).
\end{equation*}
\begin{lemma}
    Let $\mathcal{O}$ be a nilpotent orbit in $\mathfrak{s}^\prime$. The distribution on $\mathfrak{s}^\prime$ given by $f^\prime\mapsto \mu_\mathcal{O}(f^\prime)$ is $(\mathsf{H}^\prime, \eta)$-invariant. Moreover, the linear form $\mu_\mathcal{O}$ extends the $(\mathsf{H}^\prime, \eta)$-invariant distribution on $\mathcal{O}$ to an $(\mathsf{H}^\prime, \eta)$-invariant distribution on $\mathfrak{s}^\prime$.
\end{lemma}

As a consequence, we have the following. 

\begin{corollary}
A nilpotent orbit $\mathcal{O}$ supports an $(\mathsf{H}^\prime, \eta)$-invariant distribution if and only if the necessary condition in Lemma~\ref{lemma-ellipticy-NilpotentOrbit-necessary-condition} is satisfied. If $\mathcal{O}$ supports an $(\mathsf{H}^\prime, \eta)$-invariant distribution, then the distribution extends to an $(\mathsf{H}^\prime, \eta)$-invariant distribution on $\mathfrak{s}^\prime$. 
\end{corollary}

We say that a nilpotent orbit that supports an $(\mathsf{H}^\prime, \eta)$-invariant distribution (or any element of the orbit) is visible. Let $\mathcal{N}_0\subset\mathcal{N}$ be the subset of $\mathcal{N}$ consisting of visible nilpotent orbits. Then  the set 
\begin{equation*}
    \{\mu_\mathcal{O}|\mathcal{O}\subset \mathcal{N}_0\}
\end{equation*}
is an orthonormal basis of the space of $(\mathsf{H}^\prime, \eta)$-invariant distribution on $\mathfrak{s}^\prime$ supported on $\mathcal{N}$.

Let $d_\mathcal{O}=\dim_F N^+$. We have the following homogeneity property of nilpotent orbital integrals. 

\begin{lemma}
    Let $f^\prime\in \mathcal{C}_c^\infty(\mathfrak{s}^\prime)$. For any $t\in F^\times$, we put $f^\prime_t(X)=f^\prime(t^{-1}X)$. Let $\mathcal{O}\subset \mathcal{N}_0$. Then 
    \begin{equation*}
        \mathcal{\mu}_\mathcal{O}(f^\prime_t)=|t|^{d_\mathcal{O}}\eta(t)^n \mu_\mathcal{O}(f^\prime), \quad \mu_\mathcal{O}(\widehat{f^\prime_t})=|t|^{2n^2-d_\mathcal{O}}\eta(t)^n \mu_\mathcal{O}(\widehat{f^\prime}).
    \end{equation*}
\end{lemma}

\subsection{Orbital integrals}
Now we define orbital integrals on the entire space $\mathfrak{s}^\prime$, not necessarily on semisimple or nilpotent orbits. 

Let $\gamma\in \mathfrak{s}^\prime$ and let $\gamma=\gamma_s+\gamma_n$ be its Jordan decomposition where $\gamma_s$ is semisimple and $\gamma_n$ is nilpotent, both in the usual sense. Then $\gamma_s, \gamma_n\in \mathfrak{s}^\prime$. Since $\gamma_s\gamma_n=\gamma_n\gamma_s$, we have $\gamma_n\in \mathfrak{s}^\prime_{\gamma_s}$, and is nilpotent in $\mathfrak{s}^\prime_{\gamma_s}$. Assume that $\gamma_n$ is visible in $\mathfrak{s}^\prime_{\gamma_s}$ and denote its orbit by $\mathcal{O}_{\gamma_n}$. Let $f^\prime\in \mathcal{C}_c^\infty(\mathfrak{s}^\prime)$ and $h\in \mathsf{H}^\prime$. Put
\begin{equation*}
    f^\prime_1(h)=\mu_{\mathcal{O}_{\gamma_n}}(f^\prime(h^{-1}(\gamma_s+\cdot)h)).
\end{equation*}
As a function of $h\in \mathsf{H}^\prime$, $f^\prime_1$ is compactly supported on $\mathsf{H}^\prime_{\gamma_s}\backslash \mathsf{H}^\prime$.
Also, for any $y\in \mathsf{H}^\prime_{\gamma_s}$, we have $f_1^\prime(yh)=\eta(\det y))f_1^\prime(h)$. We define
an orbital integral 
\begin{equation*}
    O(\gamma, \eta, f^\prime)=\int_{\mathsf{H}^\prime_{\gamma_s}\backslash \mathsf{H}^\prime} f^\prime_1(h)\eta(\det(h))dh. 
\end{equation*}
This integral is absolutely convergent. Moreover, if the restriction of $f^\prime$ to the orbit of $\gamma$ is compactly supported, then $O(\gamma, \eta, f^\prime)$ agrees with the integral on the orbit of $\gamma$. 

Recall that by Proposition~\ref{prop-semisimple-descent}, we have an analytic Luna slice $(U, p, \psi)$ at $\gamma$. Given $f^\prime\in \mathcal{C}_c^\infty(\mathfrak{s}^\prime)$, we define
\begin{equation*}
f_{\gamma_s}^\prime(\xi)=\int_{\mathsf{H}^\prime}f^\prime(h^{-1}\psi^{-1}(\xi)h)\eta(\det(h))\alpha(h)dh, \quad \xi\in \omega_\gamma, 
\end{equation*}
where $\omega_\gamma\subset \psi(p^{-1}(\gamma))$ and $\alpha\in \mathcal{C}_c^\infty(\mathsf{H}^\prime)$ are given in \cite[Proposition 2.1]{Xue2022}. Then $f_{\gamma_s}^\prime\in \mathcal{C}_c^\infty(\mathfrak{s}_{\gamma_s})$.

\begin{lemma}
    Let $f^\prime\in \mathcal{C}_c^\infty(\mathfrak{s}^\prime)$. Then $\mu_{\mathcal{O}_{\gamma_n}}(f_{\gamma_s}^\prime)=\mathcal{O}(\gamma, \eta, f^\prime)$.
\end{lemma}

We also have the following.

\begin{lemma}
    If $\gamma_n$ is not visible in $\mathfrak{s}^\prime_{\gamma_s}$, then the orbit of $\gamma$ in $\mathfrak{s}^\prime$ does not support any $(\mathsf{H}^\prime, \eta)$-invariant distributions. 
\end{lemma}

\subsection{The germ expansion}

We have the canonical germ expansion of orbital integrals. 
\begin{proposition}
    There is a unique $(\mathsf{H}^\prime, \eta)$-invariant real-valued function $\Gamma_\mathcal{O}$ for each nilpotent orbit $\mathcal{O}\subset \mathcal{N}_0$ satisfying the following properties. 
    \begin{enumerate}[(i)]
        \item For any $f^\prime\in \mathcal{C}_c^\infty(\mathfrak{s}^\prime)$, there is an $\mathsf{H}^\prime$-invariant neighborhood $U_{f^\prime}$ of $0\in \mathfrak{s}^\prime$ such that 
        \begin{equation}
            O(\gamma, \eta, f^\prime)=\sum_{\mathcal{O}\subset \mathcal{N}_0} \Gamma_\mathcal{O}(\gamma)\mu_{\mathcal{O}}(f^\prime).
        \label{eq-ellipticy-NilpotentCone-eq6}
        \end{equation}

        \item Let $t\in F^\times$ and $\xi\in \mathfrak{s}_{\reg}^\prime$. Then 
        \begin{equation*}
            \Gamma_\mathcal{O}(t\gamma)=|t|^{-d_\mathcal{O}}\eta(t)^n \Gamma_\mathcal{O}(\gamma).
        \end{equation*}
    \end{enumerate}
\label{prop-ellipticy-canonical-Shalika-germ}
\end{proposition}

The function $\Gamma_\mathcal{O}$ is called the Shalika germ for the nilpotent orbit $\mathcal{O}$.

Now we consider the Shalika germ expansion around an arbitrary semisimple $\gamma\in \mathfrak{s}^\prime$. Let $\gamma\in \mathfrak{s}^\prime$ be a fixed semisimple element, and let $H^{\prime\prime}_\gamma=\{g\in H^\prime(F)| g^{-1}\gamma g=\gamma\}$, and $\mathsf{H}_\gamma^\prime=\mathsf{H}^\prime\cap H^{\prime\prime}_\gamma$. Let $\mathfrak{h}_\gamma^{\prime\prime}$ and $\mathfrak{h}_\gamma^\prime$ be the Lie algebras of $H^{\prime\prime}_\gamma$ and $\mathsf{H}_\gamma^\prime$ respectively. Let $\mathfrak{s}_\gamma^\prime$ be such that $\mathfrak{h}^{\prime\prime}_\gamma=\mathfrak{h}^{\prime}_\gamma\oplus \mathfrak{s}_\gamma^\prime$. Then the space $\mathfrak{s}_\gamma^\prime$ with an action of $\mathsf{H}_\gamma^\prime$ is isomorphic to $\mathfrak{s}_1^\prime \times\mathfrak{s}_2^\prime$ with an action of $\mathsf{H}_1^\prime\times \mathsf{H}_2^\prime$. We refer the reader to \cite[\S 2]{Xue2022} for the explicit description of the action of $\mathsf{H}_1^\prime\times \mathsf{H}_2^\prime$ on $\mathfrak{s}_1^\prime \times\mathfrak{s}_2^\prime$.
Let $\{\xi_1, \cdots, \xi_r\}$ be a complete set of representatives of nilpotent elements in $\mathfrak{s}_\gamma^\prime$ and $\xi_i\in \mathcal{O}_i$. Let $\Gamma_i^\gamma$ be the Shalika germ on $\mathfrak{s}^\prime_\gamma$ for the nilpotent orbit $\mathcal{O}_i$.
As an application of Proposition~\ref{prop-ellipticy-canonical-Shalika-germ}, we have
\begin{corollary}
    Let $f^\prime\in \mathcal{C}_c^\infty(\mathfrak{s}^\prime)$. There exists a neighborhood $U_{f^\prime}$ of $\gamma$ in $\mathfrak{s}_{\gamma}^\prime$ so that for any $\xi\in U_{f^\prime}\cap \mathfrak{s}^\prime_{\reg }$, we have
    \begin{equation*}
        O(\xi, \eta, f^\prime)=\sum_{i=1}^r \Gamma_i^\gamma(\xi)O(\gamma+\xi_i, \eta, f^\prime).
    \end{equation*}
\end{corollary}

\subsection{Linear independence of Shalika germs}
\label{subsection-Linear-Independence-Shalika-germs}
Similar to \cite[\S 24]{Kottwitz2005}, we have the linear independence of Shalika germs, which is closely related to the density of regular semisimple orbital integrals on $\mathfrak{s}^\prime$. 

\begin{proposition} 
We have the following results.
\begin{enumerate}
    \item The Shalika germs $\Gamma_\mathcal{O}$, $\mathcal{O}\subset \mathcal{N}_0$, are linearly independent. They are not identically zero in any neighborhood of $0$. If $\mathcal{O}=\{0\}$ is the minimal nilpotent orbit, then $\Gamma_0(\gamma)=0$ if $\gamma$ is not elliptic in $\mathfrak{s}^\prime.$
    \item The set of regular semisimple orbital integrals is weakly dense in $\mathcal{D}(\mathfrak{s}^\prime)^{\mathsf{H}^\prime, \eta}$.
\end{enumerate}
\label{ellipticity-prop-linear-independence-Shalika-germs}
\end{proposition}

As a consequence of the above proposition, by making use of Howe's finiteness theorem \cite[Theorem 6.7]{RaderRallis1996} for $\mathfrak{s}^\prime$, we have

\begin{corollary}
    The Fourier transform $\widehat{\mu_\mathcal{O}}$ of $\mu_\mathcal{O}$ is represented by a locally integrable function in $\mathfrak{s}^\prime$ for all $\mathcal{O}\subset\mathcal{N}_0$, which we will also denote by $\widehat{\mu_\mathcal{O}}$.
\label{corollary-ellipticy-representation-fourier-transform}
\end{corollary}

\subsection{Regular semisimple orbital integrals on $G^\prime$}
In this subsection, we establish results on the level of $G^\prime$. Let $\omega$ be an $(\mathsf{H}^\prime, \eta)$-invariant neighborhood of $0\in \mathfrak{s}^\prime$, and let $\Omega$ be a neighborhood of $1\in S^\prime$ such that the exponential map $\exp:\omega\to \Omega$ is defined and is a homeomorphism. Let $f^\prime\in \mathcal{C}_c^\infty(G^\prime)$. We define a function $f^\prime_\sharp\in \mathcal{C}_c^\infty(\omega)$ by requiring that
\begin{equation*}
    \int_{H^{\prime\prime}} f^\prime(gh)(\chi\tilde{\eta}^{-1})(gh)dh=f^\prime_\sharp(\gamma)
\end{equation*}
if $g\overline{g}^{-1}=\exp(\gamma)$, and we extend $f^\prime_\sharp$ to a function in $\mathcal{C}_c^\infty(\mathfrak{s}^\prime)$ by extension by zero. Let $u_1, \cdots, u_r, u_{r+1}, \cdots, u_s$ be a complete set of representatives of unipotent orbits in $G^\prime$. Let $\mathcal{O}_i$ be the nilpotent orbits in $\mathfrak{s}^\prime$ represented by $\exp^{-1}(u_i \overline{u_i}^{-1})$, and we may label them so that $\mathcal{O}_i$ is visible precisely when $1\le i \le r$. Then $u_i$ represents a unipotent orbit in $G^\prime$ which supports a left $(H^\prime, \chi^{-1}_{H^{\prime}} )$
 and right $(H^{\prime\prime}, \chi^{-1} \widetilde{\eta})$-invariant distribution precisely when $1\le i \le r$. If $f^\prime\in \mathcal{C}_c^\infty(G^\prime)$, then 
 \begin{equation*}
      O^{G^\prime}(u_i, f^\prime)=\mu_{\mathcal{O}_i}(f^\prime_\sharp).
 \end{equation*}
We call the unipotent elements $u_1, \cdots, u_r$ or their orbits visible. 
We have the Shalika germ expansion of orbital integrals on $G^\prime$. 

\begin{proposition}
    Let $f^\prime\in \mathcal{C}_c^\infty(G^\prime)$. There is a neighborhood $U_{f^\prime}\subset \Omega$ of $1\in S^\prime$ such that if $g\in G^\prime$ is regular semisimple in $G^\prime$ with $g\overline{g}^{-1}=\exp(\gamma)$ where $\gamma\in \omega$, then 
    \begin{equation*}
        O^{G^\prime}(g, f^\prime)=\sum_{i=1}^r \Gamma_{\mathcal{O}_i}(\gamma) \mu_{\mathcal{O}_i}(f^\prime_\sharp).
    \end{equation*}
\label{ellipticity-prop-orbital-integral-Shalika-germ}
\end{proposition}

As a consequence of Proposition~\ref{ellipticity-prop-orbital-integral-Shalika-germ}, we have Proposition~\ref{Prop-ellipticity-prop1}.

\subsection{The Spherical character on $G^\prime$}
\label{subsection-Spherical-character-on-Gprime}
In this subsection, we state the germ expansion for the spherical character $I_\Pi$.

Let $s^\prime=g\overline{g}^{-1}\in S^\prime$ be a semisimple element, and consider the map
\begin{equation*}
    \mathsf{H}^\prime\times H^\prime_{s^\prime}\times \mathsf{H}^\prime\to H^\prime, \quad (h_1, g, h_2)\mapsto h_1 s^\prime g h_2.
\end{equation*}
Let $U_{s^\prime}$ be the open subset of $H^\prime_{s^\prime}$ consisting of elements $g\in H^\prime_{s^\prime}$ such that the above map is submersive at $(1, g, 1)$, and let $\Omega_{s^\prime}$ be the subset $\mathsf{H}^\prime s^\prime U_{s^\prime} \mathsf{H}^\prime$. Then $U_{s^\prime}$ is a bi-$\mathsf{H}_{s^\prime}$-invariant neighborhood of 1 in $H^\prime_{s^\prime}$, and $\Omega_{s^\prime}$ is an open and bi-$\mathsf{H}^\prime$ invariant neighborhood of 1 in $H^\prime$. By the standard theory of Harish-Chandra, there exists a surjective map
\begin{equation*}
    \mathcal{C}_c^\infty(\mathsf{H}^\prime\times U_{s^\prime}\times \mathsf{H}^\prime)\to \mathcal{C}_c^\infty(\Omega_{s^\prime}), \quad \alpha\mapsto f^\prime_{\alpha},
\end{equation*}
satisfying the property that 
\begin{equation*}
    \int_{\mathsf{H}^\prime\times U_{s^\prime}\times \mathsf{H}^\prime}\alpha(h_1, g, h_2)\beta(h_1s^\prime g h_2)dh_1dg dh_2 =\int_{\Omega_{s^\prime}}f^\prime_\alpha(g)\beta(g)dg
\end{equation*}
for all $\beta\in \mathcal{C}_c^\infty(\Omega_{s^\prime})$. Then there is a unique left $\mathsf{H}_{s^\prime}$-invariant
 and right $(\mathsf{H}_{s^\prime}, \eta)$-invariant distribution $J_{s^\prime}$ on $H^\prime_{s^\prime}$, such that 
 \begin{equation*}
     I_\Pi(f_\alpha^\prime)=J_{s^\prime}(\beta_\alpha)
 \end{equation*}
 for all $\alpha\in \mathcal{C}_c^\infty( \mathsf{H}^\prime\times U_{s^\prime}\times \mathsf{H}^\prime)$, where 
 \begin{equation*}
     \beta_\alpha(g)=\int_{\mathsf{H}^\prime}\int_{\mathsf{H}^\prime}\alpha(h_1, g, h_2) \eta(\det h_2)dh_1 dh_2, \quad g\in H^\prime_{s^\prime}.
 \end{equation*}

We have the germ expansion of $I_\Pi$.
\begin{proposition}
    There are constants $c_{\mathcal{O}}$ for each visible nilpotent orbit $\mathcal{O}$  
    in the nilpotent cone $\mathcal{N}_{s^\prime}$, such that 
    \begin{equation*}
        I_\Pi(f^\prime_\alpha)=\sum_{\mathcal{O}\subset\mathcal{N}_{s^\prime}}c_\mathcal{O} \widehat{\mu_\mathcal{O}}(\beta_{\alpha, \sharp})
    \end{equation*}
    for all $\alpha\in \mathcal{C}_c^\infty( \mathsf{H}^\prime\times U_{s^\prime}\times \mathsf{H}^\prime)$. 
\label{prop-ellipticy-germ-expansion-I-Pi}
\end{proposition}

As a consequence of Proposition~\ref{prop-ellipticy-germ-expansion-I-Pi} and Corollary~\ref{corollary-ellipticy-representation-fourier-transform}, we have Proposition~\ref{prop-representability}.

\subsection{Proof of Theorem~\ref{thm-ellipticity}}
In this subsection, we prove Theorem~\ref{thm-ellipticity}. We follow the argument in \cite[Appendix B]{XueZhang2022} closely.

Recall that we have a transfer factor
\begin{equation*}
    \kappa^{G^{\prime}}(g)=\chi^{-1} (\alpha_4 ) \tilde{\eta} (\tau \alpha_2).
\end{equation*}
where
\begin{equation*}
    g \bar{g}^{-1}=\left(\begin{array}{ll}
\alpha_1 & \alpha_2 \\
\alpha_3 & \alpha_4
\end{array}\right) \in S^{\prime}(F).
\end{equation*}
Let $\widetilde{\Theta}_\Pi(g)=\kappa^{G^{\prime}}(g)\Theta_\Pi(g)$.

Now we recall the explicit description of the intertwinings from \cite[\S 4.2]{Beuzart-Plessis2018}, in the case of $\Pi$ being supercuspidal. For a linear form $\ell^\prime\in \Hom_{H^\prime}(\Pi\otimes \chi_{H^\prime}, \mathbb{C})$, we have
\begin{equation}
\ell^\prime(v) \overline{\ell^\prime(w)}=\int_{Z^\prime(F) \backslash H^\prime(F)}\langle v, \Pi(h^{-1})w\rangle \chi_{H^{\prime}} (h) dh.
\label{ellipticity-eq-linearform1}
\end{equation}
Similarly, for $\ell^{\prime\prime}\in \Hom_{H^{\prime\prime}}(\Pi\otimes \chi\widetilde{\eta}, \mathbb{C})$, we have
\begin{equation}
\ell^{\prime\prime}(v)\overline{\ell^{\prime\prime}(w)}=\int_{Z^\prime(F) \backslash H^{\prime\prime}(F)}\langle v, \Pi(h^{\prime\prime})w\rangle (\chi \widetilde{\eta})^{-1} (h^{\prime \prime} ) \mathrm{d} h^{\prime \prime}.
\label{ellipticity-eq-linearform2}
\end{equation}

\begin{lemma}
Let $\Pi$ be supercuspidal, and suppose $\Hom_{H^\prime}(\Pi\otimes \chi_{H^\prime}, \mathbb{C})\not=0$ and $\Hom_{H^{\prime\prime}}(\Pi\otimes \chi\widetilde{\eta}, \mathbb{C})\not=0$.
Let $v, w\in \Pi$ and $f^\prime(g)=\langle v, \Pi(g) w\rangle$ be the matrix coefficient of $\Pi$. Then
\begin{equation}
\label{ellipticity-lemma1-eq1}
    \kappa^{G^{\prime}}(g)O^{G^{\prime}}(g, f^{\prime})= (\chi \widetilde{\eta})^{-1}(g)\widetilde\Theta_\Pi(g) \ell^\prime(v) \overline{\ell^{\prime\prime}(w)}.
\end{equation}
\label{ellipticity-lemma1}
\end{lemma}

\begin{proof}
It suffices to prove that for any $\varphi\in \mathcal{C}_c^\infty(G^\prime)$ supported in the elliptic locus, we have
\begin{equation}
    \int_{G^\prime}  \varphi(g)\kappa^{G^{\prime}}(g)O^{G^{\prime}}(g, f^{\prime})dg= \int_{G^\prime} \varphi(g) (\chi \widetilde{\eta})^{-1}(g) \widetilde\Theta_\Pi(g) dg \cdot  \ell^\prime(v)\overline{\ell^{\prime\prime}(w)}.
\label{ellipticity-lemma1-eq2}
\end{equation}
Note that the right-hand side of \eqref{ellipticity-lemma1-eq2} is equal to $I_\Pi(\varphi \kappa^{G^\prime}(\chi\widetilde{\eta})^{-1})\cdot  \ell^\prime(v)\overline{\ell^{\prime\prime}(w)}$.

Since $\left(H^{\prime} \times H^{\prime \prime}\right)_g$ is an anisotropic torus modulo the split center $Z^\prime$ of $G^\prime$, up to a non-zero constant depending only on the choice of the measures, the orbital integral $O^{G^{\prime}}(g, f^{\prime})$ is equal to
\begin{equation*}
    \int_{Z^\prime\left(F\right) \backslash\left(H^{\prime} \times H^{\prime \prime}\right)\left(F\right)} f^{\prime} (h^{-1} g h^{\prime \prime} )\left(\chi_{H^{\prime}} \chi^{-1} \widetilde{\eta}^{-1}\right)(h)(\chi \widetilde{\eta})^{-1} (h^{-1} g h^{\prime \prime} ) \mathrm{d} h \mathrm{~d} h^{\prime \prime}.
\end{equation*}
Since $\varphi$ is supported in the elliptic locus, the left-hand side of \eqref{ellipticity-lemma1-eq2} is equal to
\begin{equation*}
    \int_{G^\prime} \int_{Z^\prime\left(F\right) \backslash\left(H^{\prime} \times H^{\prime \prime}\right)\left(F\right)} \varphi(g) \kappa^{G^\prime}(g)\langle v, \Pi(h^{-1} g h^{\prime \prime})w\rangle \left(\chi_{H^{\prime}} \chi^{-1} \widetilde{\eta}^{-1}\right)(h)(\chi \widetilde{\eta})^{-1} (h^{-1} g h^{\prime \prime} ) \mathrm{d} h \mathrm{~d} h^{\prime \prime}dg.
\end{equation*}
This integral is absolutely convergent. Changing the order of integration, we conclude that the left-hand side of \eqref{ellipticity-lemma1-eq2} equals
\begin{equation*}
    \int_{Z^\prime\left(F\right) \backslash\left(H^{\prime} \times H^{\prime \prime}\right)\left(F\right)} \langle v, \Pi(h^{-1})\Pi(\overline{ \varphi \kappa^{G^\prime}(\chi \widetilde{\eta})^{-1}}) \Pi(h^{\prime \prime})w\rangle  \left(\chi_{H^{\prime}} \chi^{-1} \widetilde{\eta}^{-1}\right)(h)(\chi \widetilde{\eta})^{-1} (h^{-1}  h^{\prime \prime} ) \mathrm{d} h \mathrm{~d} h^{\prime \prime}.
\end{equation*}
This simplifies to
\begin{equation}
    \int_{Z^\prime\left(F\right) \backslash\left(H^{\prime} \times H^{\prime \prime}\right)\left(F\right)} \langle v, \Pi(h^{-1})\Pi(\overline{ \varphi \kappa^{G^\prime}(\chi \widetilde{\eta})^{-1}}) \Pi(h^{\prime \prime})w\rangle  \chi_{H^{\prime}} (h)(\chi \widetilde{\eta})^{-1} (h^{\prime \prime} ) \mathrm{d} h \mathrm{~d} h^{\prime \prime}.
\label{ellipticity-lemma1-eq3}
\end{equation}
Since $\Pi$ is admissible, there are vectors $v_1, \cdots, v_r$ and $w_1, \cdots, w_r$ in $\Pi$ such that
\begin{equation*}
    \Pi(\overline{ \varphi \kappa^{G^{\prime}}  (\chi \widetilde{\eta})^{-1}})(v)=\sum_{i=1}^r \langle v, v_i\rangle w_i.
\end{equation*}
Then  
\begin{equation*}
    \eqref{ellipticity-lemma1-eq3}=\sum_{i=1}^{r} \int_{Z^\prime\left(F\right) \backslash\left(H^{\prime} \times H^{\prime \prime}\right)\left(F\right)} \langle v, \Pi(h^{-1}) w_i\rangle \langle v_i, \Pi(h^{\prime\prime})w\rangle  \chi_{H^{\prime}}^{-1} (h)(\chi \widetilde{\eta})^{-1} (h^{\prime \prime} ) \mathrm{d} h \mathrm{~d} h^{\prime \prime}.
\end{equation*}
Now we apply \eqref{ellipticity-eq-linearform1} and \eqref{ellipticity-eq-linearform2} to conclude that the left-hand side of \eqref{ellipticity-lemma1-eq2} is equal to
\begin{equation}
\label{ellipticity-lemma1-eq4}
    \sum_{i=1}^{r} \ell^\prime(v) \overline{\ell^\prime(w_i)}\ell^{\prime\prime}(v_i)\overline{\ell^{\prime\prime}(w)}
\end{equation}
On the other hand, we have
\begin{equation*}
\begin{split}
I_\Pi(\varphi \kappa^{G^\prime}(\chi\widetilde{\eta}^{-1})
=& \sum_{u} \overline{\ell^\prime( \Pi(\overline{ \varphi \kappa^{G^\prime}(\chi\widetilde{\eta}})^{-1}) u)\overline{\ell^{\prime\prime}(u)} } \\
=& \sum_{u} \overline{\ell^\prime(\sum_{i=1}^r \langle u, v_i\rangle w_i   )\overline{\ell^{\prime\prime}(u)}}  \\
=& \sum_{u} \sum_{i=1}^r \overline{\langle u, v_i\rangle} \overline{\ell^\prime(  w_i   )} \ell^{\prime\prime}(u)  \\
=& \sum_{i=1}^r \overline{ \ell^\prime(w_i) } \ell^{\prime\prime}(v_i).
\end{split}
\end{equation*}
Thus the right-hand side of \eqref{ellipticity-lemma1-eq2} is equal to \eqref{ellipticity-lemma1-eq4}. The proof is now complete. 
\end{proof}

\begin{proof}[Proof of Theorem~\ref{thm-ellipticity}]
    Let $\gamma$ be in a small neighborhood of $0\in \mathfrak{s}^\prime$, and take $g\in G^\prime$ such that $g\overline{g}^{-1}=\exp(\gamma)$. Similar to \cite[Theorem 7.11]{RaderRallis1996}, we have a character expansion
    \begin{equation}
    \widetilde{\Theta}_\Pi(g)=\sum_{\mathcal{O}}c_\mathcal{O} \widehat{\mu_\mathcal{O}}(\gamma),
    \label{ellipticity-lemma1-eq5}
    \end{equation}
    where $\widehat{\mu_\mathcal{O}}$ is a locally integrable function on $\mathfrak{s}^\prime$ which represents the Fourier transform of the orbital integral $\mathcal{\mu}_\mathcal{O}$ as in Corollary~\ref{corollary-ellipticy-representation-fourier-transform}.
    Note that $g$ is elliptic if and only if $\gamma$ is elliptic. To prove that $\Pi$ is $(H^\prime, H^{\prime\prime})$-elliptic, it suffices to show that $\widetilde{\Theta}_\Pi(g)\not=0$ for some elliptic $g\in G^\prime$ which is sufficiently close to $1$. Since $\mathcal{O}=\{0\}$ is the only nilpotent orbit with $\widehat{\mu_\mathcal{O}}(t\gamma)=\widehat{\mu_\mathcal{O}}(\gamma)$ for all $\gamma\in \mathfrak{s}^\prime$ and $t\in F^\times$, it suffices to show that $c_0\not=0$. 

    Let $v, w\in \Pi$ and $f^\prime(g)=\langle v, \Pi(g) w\rangle$ be the matrix coefficient of $\Pi$ such that 
    \begin{equation*}
        \int_{Z^\prime(F) \backslash H^\prime(F)}f^\prime (h^{-1})  \chi_{H^{\prime}} (h) dh \cdot \int_{Z^\prime(F) \backslash H^{\prime\prime}(F)}f^\prime(h^{\prime\prime})  (\chi \widetilde{\eta})^{-1} (h^{\prime \prime} ) \mathrm{d} h^{\prime \prime} \not=0.
    \end{equation*}
    By \eqref{ellipticity-eq-linearform1} and \eqref{ellipticity-eq-linearform2}, the above inequality is equivalent to \begin{equation*}
        \ell^\prime(v) \overline{\ell^\prime(w)}\ell^{\prime\prime}(v)\overline{\ell^{\prime\prime}(w)}\not=0.
    \end{equation*}
    Now we consider both sides of \eqref{ellipticity-lemma1-eq1} when $g$ is sufficiently close to $1$. By the Shalika germ expansion of orbital integral (see Proposition~\ref{ellipticity-prop-orbital-integral-Shalika-germ}) and the character expansion of spherical character (see \eqref{ellipticity-lemma1-eq5}), we have
    \begin{equation*}
         \sum_{\mathcal{O}} \kappa^{G^\prime}(g)\Gamma_{\mathcal{O}}(\gamma) \mu_{\mathcal{O}}(f^\prime_\sharp) = \sum_{\mathcal{O}}(\chi \widetilde{\eta})^{-1}(g) 
 c_\mathcal{O} \widehat{\mu_\mathcal{O}}(\gamma)  \ell^\prime(v) \overline{\ell^{\prime\prime}(w)}.
    \end{equation*}
    The only terms in both sides of the expansion that are invariant under the scaling $\gamma\mapsto t\gamma$ are the terms corresponding to $\mathcal{O}=\{0\}$. It follows from the homogeneity property of $\Gamma_{\mathcal{O}}$ and $\widehat{\mu_{\mathcal{O}}}$ that
    \begin{equation}
        \kappa^{G^\prime}(g)\Gamma_0(\gamma) \mu_{0}(f^\prime_\sharp) = (\chi \widetilde{\eta})^{-1}(g) 
 c_0 \widehat{\mu_0}(\gamma)  \ell^\prime(v) \overline{\ell^{\prime\prime}(w)}.
  \label{ellipticity-lemma1-eq6}
    \end{equation}
    By our choice of $f^\prime$, we have 
    \begin{equation*}
        \mu_{0}(f^\prime_\sharp) = \int_{H^{\prime\prime}(F)} f^\prime(h^{\prime\prime})(\chi\tilde{\eta})(h^{\prime\prime})dh^{\prime\prime}\not=0.
    \end{equation*}
    If $c_0=0$, then from \eqref{ellipticity-lemma1-eq6} we deduce that $\Gamma_0(\gamma)=0$ if $\gamma$ is elliptic in a neighborhood of $0$. By Proposition~\ref{ellipticity-prop-linear-independence-Shalika-germs}, we also know that $\Gamma_0(\gamma^\prime)=0$ if $\gamma^\prime$ is not elliptic. Hence $\Gamma_0$ is identically zero in a neighborhood of 0. This contradicts Proposition~\ref{ellipticity-prop-linear-independence-Shalika-germs}. Therefore, we conclude that $c_0\not=0$. This finishes the proof of Theorem~\ref{thm-ellipticity}.
\end{proof}

\bibliographystyle{alpha}
\bibliography{References}

\end{document}